\documentclass[english,11pt]{article}
\usepackage{tikz}
\usetikzlibrary{decorations.pathreplacing}
\usepackage{enumerate}
\usepackage[shortlabels]{enumitem}
\usepackage{verbatim}
\usepackage[margin=1in]{geometry}
\usepackage{amssymb}
\usepackage{caption}
\usepackage{bbm}
\usepackage{amsthm}
\usepackage{enumitem}

\usepackage{scalerel}    
\usepackage{stmaryrd}   

\makeatletter
\DeclareRobustCommand\widecheck[1]{{\mathpalette\@widecheck{#1}}}
\def\@widecheck#1#2{%
    \setbox\z@\hbox{\m@th$#1#2$}%
    \setbox\tw@\hbox{\m@th$#1%
       \widehat{%
          \vrule\@width\z@\@height\ht\z@
          \vrule\@height\z@\@width\wd\z@}$}%
    \dp\tw@-\ht\z@
    \@tempdima\ht\z@ \advance\@tempdima2\ht\tw@ \divide\@tempdima\thr@@
    \setbox\tw@\hbox{%
       \raise\@tempdima\hbox{\scalebox{1}[-1]{\lower\@tempdima\box
\tw@}}}%
    {\ooalign{\box\tw@ \cr \box\z@}}}
\makeatother

\usepackage{rotating,centernot,cancel}    

\usepackage{amsmath}
\usepackage{tocloft}
\usepackage{float}

\usepackage{tcolorbox}   
\PassOptionsToPackage{hyphens}{url}\usepackage{hyperref}
\usepackage{babel}
\theoremstyle{plain}

\input amssym.def
\input amssym.tex

\def\beq{\begin{equation}}
\def\eeq{\end{equation}}
\def\beqn{\begin{eqnarray}}
\def\eeqn{\end{eqnarray}}

\usepackage{mathtools}
\DeclarePairedDelimiter\floor{\lfloor}{\rfloor}

\newtheorem{theorem}{Theorem}[section]

\newtheorem{lemma}[theorem]{Lemma} 
\newtheorem*{lemma*}{Lemma}
\newtheorem{proposition}[theorem]{Proposition} 

\theoremstyle{remark}
\newtheorem{remark}[theorem]{Remark}

\theoremstyle{definition}

\numberwithin{figure}{section}

\def\R{\mathbb R}

\def\Z{\mathbb Z}

\def\to{\rightarrow}
\def\dlim[#1][#2]{\lim_{#1 \to #2, #1 \neq #2}}

\def\Var{\textup{$\mathbb{V}$ar}}
\def\Exp{\textup{Exp}}

\newcommand{\be}{\begin{equation}}
\newcommand{\ee}{\end{equation}}

\providecommand{\abs}[1]{\vert#1\vert}

\newcommand{\fl}[1]{\lfloor{#1}\rfloor} 

\newcommand\bbullet{{{\scaleobj{0.6}{\bullet}}}} 
\newcommand\mydots{\hbox to 1em{.\hss.\hss.}}

\def\w{\omega} 

\def\wt{\widetilde}    \def\wc{\widecheck}

\def\bgeod#1#2{\mathbf{b}^{#1, #2}}

\newcommand{\myfootnote}[1]{
    \renewcommand{\thefootnote}{}
    \footnotetext{\scriptsize#1}
    \renewcommand{\thefootnote}{\arabic{footnote}}
}

\allowdisplaybreaks

\title{Lower bound for large transversal fluctuations\\ in exactly solvable KPZ models}
\author{Xiao Shen\thanks{\scriptsize{Department of Mathematics, University of Utah, Utah, USA. \texttt{xiao.shen@utah.edu}}}}
\date{}
\begin{document}
\maketitle

\begin{abstract}
The study of transversal fluctuation of the optimal path has been a crucial aspect of the Kadar-Parisi-Zhang (KPZ) universality class. In this paper, we establish a new probability lower bound, with optimal exponential order, for the rare event in which a given level of the optimal path has a large transversal fluctuation. We present our results in both zero and positive temperature settings. The previously known lower bounds were obtained in zero temperature models, and they hold for the maximum transversal fluctuation along the entire geodesic \cite{ham-sar-20} or the starting portion of the geodesic at a local scale \cite{aga-23}. Our result improves upon these as now the rare event can demand where the large fluctuation occurs exactly along the optimal path, on both local and global scales.
Our proof utilizes the coupling method: we first obtain a version of the estimate in the semi-infinite setting using duality and then transfer the result to finite paths using planar monotonicity. Our method differs from the previous works \cite{aga-23, ham-sar-20}, and in fact, we do not require fine information about the left tail moderate deviation, which played a crucial role in \cite{aga-23, ham-sar-20}.
\end{abstract}

\myfootnote{Date: \today}
\myfootnote{2010 Mathematics Subject Classification. 60K35, 	60K37}
\myfootnote{Key words: last-passage percolation, directed polymer, transversal fluctuation, Kardar-Parisi-Zhang, Busemann function.}

\section{Introduction and results}\label{intro}

The exploration of universality is at the core of probability and statistics. It refers to a remarkable phenomenon observed in many random growths, where diverse systems with different underlying microscopic details exhibit universal macroscopic behavior. 
One classic example illustrating universality is the \textit{central limit theorem} (CLT), which states that the large-scale behavior of the centered sum becomes independent of the specific distribution of the individual summands.

Another strikingly different universality class was predicted by Kardar, Parisi, and Zhang \cite{Kar-Par-Zha-86} in 1986, known as the \textit{KPZ universality class}. This new universal behavior was anticipated to emerge across a wide range of stochastic models with spatial dependence.
Extensive computer simulations and empirical laboratory experiments have collectively suggested that the KPZ universality class is extremely rich. It encompasses percolation models, directed polymers, interacting particle systems, random tilings, stochastic partial differential equations, and more.

An important subset within the KPZ universality class comprises stochastic models that characterize optimal paths navigating through a random environment. Central to this domain is the analysis of transversal fluctuation, a measure of the deviation of optimal paths from certain reference lines. Understanding transversal fluctuations has proven crucial in unraveling other significant random geometric properties and space-time profiles associated with random growth phenomena. The main focus of this paper is a new probability lower bound for when the optimal paths exhibit an unusually large transversal fluctuation. 

We present our results for two representative KPZ lattice models: the corner growth model (CGM), also known as exponential last-passage percolation (LPP), and the inverse-gamma polymer. These models represent the zero-temperature and positive-temperature KPZ models, respectively. This classification is because, under certain scaling, the CGM can be viewed as a limit of the inverse-gamma polymer. Thus, taking this limit is often called scaling the temperature parameter from a positive value to zero.

\subsection{The corner growth model}

The \textit{corner growth model} is defined on the integer lattice  $\mathbb{Z}^2$.  Let $\{\omega_{\mathbf z}\}_{\mathbf z\in \Z^2}$ be i.i.d.~Exp(1) random variables  associated with the integer lattice.
Given two coordinatewise-ordered points ${\mathbf{x}}, {\mathbf{y}}$ of $ \mathbb{Z}^2$, the last-passage value $G_{\mathbf{x}, \mathbf{y}}$ is  the maximum accumulation of weights along each directed path from $\mathbf{x}$ to $\mathbf{y}$:
\begin{equation}\label{sec2G}G_{\mathbf x,\mathbf y} = \max_{{\boldsymbol\gamma} \in \mathbb{X}_{\mathbf x,\mathbf y}} \sum_{\mathbf{z} \in {\boldsymbol\gamma}} \omega_{\mathbf z}
\end{equation}
where $\mathbb{X}_{\mathbf x,\mathbf y}$ is the collection of paths ${\boldsymbol\gamma}=({\boldsymbol\gamma}_k)_{k=0}^{|\mathbf y-\mathbf x|_1}$ that satisfy ${\boldsymbol\gamma}_0=\mathbf x$, ${\boldsymbol\gamma}_{|\mathbf y-\mathbf x|_1}=\mathbf y$, and ${\boldsymbol\gamma}_{k+1}-{\boldsymbol\gamma}_k\in\{\mathbf e_1, \mathbf e_2\}$.
Moreover, the almost surely unique maximizing path for $G_{\mathbf{x}, \mathbf{y}}$ is often referred to as the \textit{geodesic}.

Fix the starting point $\mathbf{x} = {(0,0)}$ and the endpoint $\mathbf{y} = (n,n)$, the {\it transversal fluctuation} of the geodesic of $G_{(0,0), (n,n)}$ quantifies the deviation of the geodesic path from the diagonal line between $(0,0)$ and $(n,n)$. If we look at the middle of the geodesic, it is expected to deviate away from the diagonal on the scale $n^{2/3}$. It has been conjectured that the exponent $2/3$ should remain valid for a broad range of weights beyond the exponential distribution.

\subsection{The inverse-gamma polymer}\label{inv_intro}

Recall that a random variable $X$ has the inverse-gamma distribution with shape parameter $\mu\in(0,\infty)$, abbreviated as $X\sim \text{Ga}^{-1}(\mu)$, if $X^{-1}$ has the gamma distribution with the shape parameter $\mu$. 
The inverse-gamma polymer is also defined in $\mathbb{Z}^2$. Let $\{\omega_\mathbf z\}_{\mathbf z\in \mathbb{Z}^2}$ be i.i.d.~inverse-gamma distributed random variables with a fixed shape parameter $\mu \in (0, \infty)$. To maintain consistency in the notation with the CGM, we set $\mu=1$, although we note that the exact same arguments can also be applied for general $\mu\in (0, \infty)$. For two coordinatewise-ordered vertices $\mathbf x$ and $\mathbf y$ of $\mathbb{Z}^2$, the \textit{point-to-point partition function} is defined by 
\begin{equation}\label{def_part}
Z_{\mathbf x, \mathbf y} = \sum_{\boldsymbol\gamma \in \mathbb{X}_{\mathbf x, \mathbf y}} \prod_{\mathbf{z}\in \boldsymbol\gamma} \omega_{\mathbf z}.
\end{equation}
We use the convention $Z_{\mathbf x,\mathbf y} =  0$ if $\mathbf x\leq \mathbf y$ fails. 
Moreover, the \textit{free energy} is defined as $\log Z_{\mathbf x, \mathbf y}$.

The \textit{quenched polymer measure} is a probability measure on the set $\mathbb{X}_{\mathbf x,\mathbf y}$ and is defined by 
$$Q_{\mathbf x,\mathbf y}\{\boldsymbol\gamma\} = \frac{1}{Z_{\mathbf x,\mathbf y}} \prod_{\mathbf{z}\in \boldsymbol\gamma} \omega_{\mathbf z} \qquad \text{ for $\boldsymbol\gamma \in \mathbb{X}_{\mathbf x,\mathbf y}$}.$$
Compared to the CGM, where the geodesic can be thought of as a (random) Dirac-delta measure on $\mathbb{X}_{\mathbf x, \mathbf y}$ which is supported on the geodesic of $G_{\mathbf x, \mathbf y}$, the quenched polymer measure $Q_{\mathbf x, \mathbf y}$ is a (random) probability measure on $\mathbb{X}_{\mathbf x, \mathbf y}$ which is concentrated around the maximizing path of the free energy $\log Z_{\mathbf x, \mathbf y}$. A second layer of randomness enters the picture as we are interested in the geometry of the sampled paths from this quenched polymer path measure $Q_{\mathbf x, \mathbf y}$.

\subsection{Main results}

In Theorem \ref{main1} and Theorem \ref{main2} below, we present our results concerning a probability lower bound for the rare event in which the geodesic or the polymer path has an unusually large transversal fluctuation. For a review of the literature related to this, please refer to Section \ref{contri}.

For each ${\boldsymbol\gamma} \in \mathbb{X}_{(0,0), (n,n)}$ and $0 \leq r \leq n$, let $\mathbf{v}_r^{\textup{min}}({\boldsymbol\gamma})$ denote the lattice point with the minimum $\ell^1$-norm among the intersection of ${\boldsymbol\gamma}$ and the horizontal line $y = r$. The choice of the horizontal line over the anti-diagonal line is made to simplify the notation in our proof. 
The results for the CGM and the inverse-gamma polymer are stated below in Theorem \ref{main1} and Theorem \ref{main2} respectively.
\begin{theorem}\label{main1}
There exist positive constants $C, r_0, n_0, c_0$ such that for each $n\geq n_0, r_0 \leq r \leq n/2$ and $1 \leq t \leq c_0 r^{1/3}$, it holds that
$$\mathbb{P}\Big(\mathbf{e}_1 \cdot \mathbf{v}^{\textup{min}}_r \big(\textup{the geodesic of $G_{(0,0), (n,n)}$}\big) - r \geq tr^{2/3}\Big) \geq e^{-Ct^3}.$$ 
\end{theorem}
\begin{theorem}\label{main2}
There exist positive constants $C_1, C_2, C_3, r_0, n_0, c_0$ such that for each $n\geq n_0, r_0 \leq r \leq n/2$ and $1 \leq t \leq c_0 r^{1/3}$, it holds that
$$\mathbb{P}\Big(Q_{(0,0), (n,n)}\big\{ {\boldsymbol\gamma} \in \mathbb{X}_{(0,0), (n,n)} \;: \: \mathbf{e}_1 \cdot \mathbf{v}^{\textup{min}}_r ({\boldsymbol\gamma}) - r \geq tr^{2/3} \big\} \geq 1-{e^{-C_1 t^2 n^{1/3}}}\Big) \geq e^{-C_2t^3},$$
and this implies
$$\mathbb{E}\Big[Q_{(0,0), (n,n)}\big\{ {\boldsymbol\gamma} \in \mathbb{X}_{(0,0), (n,n)} \;: \: \mathbf{e}_1 \cdot \mathbf{v}^{\textup{min}}_r ({\boldsymbol\gamma}) - r \geq tr^{2/3} \big\}\Big] \geq e^{-C_3t^3}.$$
\end{theorem}

\begin{remark}
When $r = n/2$ we obtain the result for the midpoint transversal fluctuation. Note because we only require $r$ to be in the range $r_0 \leq r \leq n/2$, we can also look at local fluctuation. This means that we can fix the value of $r$ and allow $n$ to grow arbitrarily large. 
\end{remark}
\begin{remark}
The exponential orders in our lower bounds are optimal as the probabilities above are also upper bounded by $e^{-Ct^3}$, see Section \ref{contri}.
\end{remark}

To prove the theorems, we will start with Theorem \ref{main1} for the CGM in Section \ref{pf_main}. This approach is more intuitive because the existence of geodesics and path monotonicity play crucial roles in our argument. Then, in Section \ref{pf_main2}, we will prove Theorem \ref{main2}, providing the details of how path monotonicity applies in the positive temperature setting, even though there are no geodesics. Lastly, in Section \ref{sec_lem}, we will prove several lemmas for which the arguments apply to both the CGM and the inverse-gamma polymer. This similarity arises from the analogous concentration inequalities for the last-passage value in the CGM and the free energy in the inverse-gamma polymer.


\subsection{Our contribution to the related literature} 
\label{contri}

We start our discussion with the zero temperature models. 
In the seminal work of \cite{kurt_fluc}, it was shown that the transversal fluctuation of the geodesic between $(0,0)$ and $(n,n)$ is of order $n^{2/3+o(1)}$ as $n$ tends to infinity. This, in turn, verified the transversal fluctuation exponent $2/3$.  Subsequently, we narrow our focus to geodesics exhibiting substantial transversal fluctuations of the order $tn^{2/3}$ for large values of $t$. For the study of small transversal fluctuations of the size $\delta n^{2/3}$ for small $\delta$, we refer to the recent work \cite{small_deviation_LPP}.

Let $\mathbf{v}^{\textup{max}}_{n/2}(G_{0,n})$ represent the lattice point with the maximum $\ell^1$-norm among the intersection points of the geodesic from ${(0,0)}$ to $(n, n)$ and the horizontal line $y = n/2$.
Let $R_{0,n}^{tn^{2/3}}$ denotes the parallelogram spanned by the four corners $(0,0)\pm(tn^{2/3}, -tn^{2/3})$ and $ (n,n) \pm(tn^{2/3}, -tn^{2/3})$. We define the probabilities of two events as follows:
\begin{align*}(\textup{Mid}) &= \mathbb{P}  \big(|  \mathbf{e}_1 \cdot \mathbf v^{\textup{max}}_{n/2} (G_{0,n}) - n/2| > tn^{2/3}\big)\\
(\textup{Tube}) &= \mathbb{P}\big(\textup{the geodesic of $G_{(0,0), (n,n)}$ exits the tube } R_{0,n}^{tn^{2/3}}\big).
\end{align*}
By definition, $(\textup{Mid})  \leq (\textup{Tube})$.
\begin{itemize}
\item \textbf{Upper bounds:} 
An initial polynomial upper bound of $Ct^{-3}$ for (Mid) was established in \cite[Theorem 2.5]{poly2}. Subsequently, exponential upper bounds of $e^{-Ct}$ and $e^{-Ct^3}$ were derived in \cite[Lemma 11.3]{slowbondproblem} and \cite[Lemma C.10]{timecorrflat}, respectively. The exponential upper bound for (Mid) is then leveraged through an iterative argument to obtain an upper bound for $(\textup{Tube})$, see \cite[Lemma 11.1]{slowbondproblem}  and \cite[Lemma C.8]{timecorrflat}.
Thus, it holds that 
$$(\textup{Mid})  \leq (\textup{Tube}) \leq e^{-Ct^3}.$$
\item \textbf{Lower bounds:} 
The  existing lower bound, as established in \cite[Proposition 1.4]{ham-sar-20}, asserts that $(\textup{Tube})$ is lower bounded by 
$e^{-C't^3}.$ The key argument there is to construct a rare event where a path exiting outside the tube $R^{tn^{2/3}}_{0,n}$ possesses a greater passage value than the passage values of all paths that lie entirely inside  $R^{tn^{2/3}}_{0,n}$. This, in turn, implies that the geodesic must exit the tube $R^{tn^{2/3}}_{0,n}$. However, this argument does not provide information about the specific location of the large transversal fluctuation, rendering it inapplicable to (Mid).
To complement this, our result addresses the final missing piece in the tail bounds, as now we also have 
$$e^{-C't^3} \leq (\textup{Mid}) \leq (\textup{Tube}).$$
\end{itemize}

In addition to the global transversal fluctuation, another critical aspect involves studying the local transversal fluctuation of the geodesic near its starting point. Given $r \ll n$, the focus in this context is on the point where the geodesic between $(0,0)$ and $(n,n)$ intersects the line $x+y = r$, measuring its deviation from $(r,r)$ on the scale of $r^{2/3}$. This is analogous to (Mid). A corresponding probability to (Tube) considers the segment of the geodesic from $(0,0)$ to just before it crosses $x+y = r$, measuring its deviation from the diagonal line.

Probability upper bounds for the local fluctuation analogous to (Mid) and (Tube) were derived in \cite[Theorem 3]{ubcoal} and \cite[Proposition 2.1]{balazs2023geodesic} respectively. The available lower bound again only pertains to the rare event corresponding to (Tube), appearing in \cite{aga-23}. As mentioned in \cite[Remark 5]{aga-23}, the argument used there is similar to \cite{ham-sar-20}, thus the rare event can not precisely identify the location of the large fluctuation. As mentioned before, our result also addresses this limitation.

In addition to these tail bounds, an exact formula for the probability of the LPP geodesics passing through a given point was derived in \cite{liu-geo}. However, as pointed out by
the author, it has remained difficult to extract tail bounds from the formula due to its complicated form. In addition, the corresponding exact formula is not currently known for any positive temperature KPZ models, such as the inverse-gamma polymer addressed in our paper.

Furthermore, versions of the lower bound have been obtained in the semi-infinite setting for both the CGM and the inverse-gamma polymer \cite{ ras-sep-she-, seppcoal}, using coupling methods. This approach forms the basis of our paper, and the novelty lies in incorporating the dual semi-infinite geodesics/polymers into the picture and extending previous arguments to control both the primal and dual geodesics/polymers simultaneously. This, in turn, enables us to derive the result for finite geodesics/polymers as stated in Theorem \ref{main1} and Theorem \ref{main2}.

A few months after we posted the first version of this paper, another related work \cite{agarwal2024sharp} identified a precise expression for (Mid), including the leading order constant in the exponential power. Although our results do not identify this constant, we highlight two key differences that make our results remain significant:

\begin{itemize}
\item The leading constant in the exponential found in \cite{agarwal2024sharp} applies to the midpoint of the geodesic, when \(r = n/2\). The explicit constant for positions away from the midpoint remains unknown (see Section 1.1 of \cite{agarwal2024sharp}). Our bound applies to any fixed level, i.e.\ \(r_0 \leq r \leq n-r_0\).

\item Our results also cover the positive temperature setting, i.e.\ the inverse-gamma polymer, whereas the percolation method from \cite{agarwal2024sharp} was developed for the CGM.
\end{itemize}

\section{Preliminaries}
This section covers aspects of the models used in our proof. Section \ref{sec_stat} introduces the increment-stationary versions of the CGM and the inverse-gamma polymer, and we recall important known results, such as the variance and the exit time bounds. Section \ref{semiinf} and Section \ref{Bus+Poly} introduce the Busemann function in the CGM and inverse-gamma polymer.  They are then used to study the primal and dual semi-infinite geodesics/polymers. Additionally, we recall two known results about the disjointedness of the primal and dual semi-infinite geodesics/polymers and their connections to the increment-stationary models.


\subsection{Increment-stationary models}\label{sec_stat}

We start by defining an increment-stationary LPP process starting from a horizontal boundary, which is a model defined on a half space $\mathbb{Z}\times \mathbb{Z}_{\geq 0}$ instead of the entire $\mathbb{Z}^2$. Without the loss of generality, we will fix the starting point to be the origin, and the horizontal boundary is located along the horizontal line $y=0$. 

Let $\{\omega_{\mathbf z}\}_{\mathbf z\in \mathbb{Z}\times \mathbb{Z}_{>0}}$ be a collection of i.i.d.\ $\Exp(1)$ distributed random variables. Fix $\rho \in (0,1)$,  let $\{Y_{j}\}_{j\in \mathbb{Z}}$ be a collection of i.i.d.\ $\Exp(\rho)$ distributed random variables, and they will serve as the boundary weights. It will be helpful to think of the boundary weights $Y_j$ as being attached to the unit edges $[\![(j-1,0), (j, 0)]\!]$ along the line \(y=0\). With $\{Y_{j}\}_{j\in \mathbb{Z}}$ given, define $\{h_i\}_{i \in \mathbb{Z}}$ to be $h_0 = 0$ and 
$$h_i = \begin{cases}
\sum_{j=1}^i Y_j \qquad & \textup{ if } i \geq 1\\
\sum_{j=0}^{|i|-1} -Y_{-j} \qquad & \textup{ if } i \leq -1.
\end{cases}
$$
Recall the bulk last-passage value $G$ defined in \eqref{sec2G}, then the last-passage value starting from $\Exp(\rho)$ weights along the horizontal boundary will be 
\begin{equation}\label{Gh}
G^{\rho}_{\mathbf 0, \mathbf x} = \begin{cases} \max_{i \in \mathbb{Z}} \big\{h_i + G_{(i,1), \mathbf x}\big\} \qquad & \textup{ if } \mathbf{x} \cdot \mathbf{e}_2 \geq 1\\
h_{\mathbf{x} \cdot \mathbf{e}_1}\qquad & \textup{ if } \mathbf{x} \cdot \mathbf{e}_2 = 0.
\end{cases}
\end{equation}
We will view the unique geodesic in this model as an upright path starting from the horizontal line $y=0$ and ending at $\mathbf{x}$, i.e. an element of  $\cup_{k\in \mathbb{Z}} \mathbb{X}_{(k,0), \mathbf{x}}$.

The increment-stationary polymer is defined in a very similar fashion. Fix $\rho \in (0,1)$, let $\{Y_{j}\}_{j\in \mathbb{Z}}$ be a collection of i.i.d.\ $\text{Ga}^{-1}(\rho)$ distributed random variables. Define $h_0 = 1$ and 
$$h_i = \begin{cases}
\prod_{j=0}^i Y_{j} \qquad & \textup{ if } i \geq 1\\
\prod_{j=0}^{|i|-1} Y_{-j}^{-1} \qquad & \textup{ if } i \leq -1.
\end{cases}
$$
The free energy with $\text{Ga}^{-1}(\rho)$ distributed 
 boundary  weights is defined by 
\begin{equation}\label{free}
\log Z^{\rho}_{\mathbf 0,\mathbf x} = 
\begin{cases}
\log \Big(\sum_{i \in \mathbb{Z}} h_i \cdot Z_{(i,1), \mathbf x}\Big) \qquad &\textup{ if } \mathbf{x}\cdot \mathbf{e}_2 \geq 1\\
\log h_{ \mathbf{x}\cdot \mathbf{e}_1}\qquad &\textup{ if } \mathbf{x}\cdot \mathbf{e}_2 =0\\
\end{cases}
\end{equation}
The quenched polymer measure in this model can be viewed as a probability measure on $\cup_{k\in \mathbb{Z}} \mathbb{X}_{(k,0), \mathbf{x}}$. For a particular up-right path ${\boldsymbol\gamma}$ starting at $(k,0)$, the quenched polymer measure is defined by  $$Q^\rho_{(0,0), \mathbf{x}}({\boldsymbol\gamma}) = \frac{ h_k \cdot \prod_{\mathbf{z}\in \boldsymbol\gamma} \omega_{\mathbf z} }{Z^{\rho}_{\mathbf 0,\mathbf x}}.$$

\begin{remark}
In the past, the increment-stationary models have often been defined with a southwest boundary, aligned along the $\mathbf{e}_1$- and $\mathbf{e}_2$-axes, rather than a horizontal boundary. Notably, there is no real difference as there is a coupling between the two definitions such that that the last-passage times or the free energies from $(0,0)$ to $\mathbf{x} \in \mathbb{Z}{\geq 0}$ are identical. More generally, one could actually define the increment-stationary models with its boundary along any down-right path, as appeared in \cite[Appendix B]{opt_exit} and \cite[Section 5]{coalnew}.
\end{remark}

\begin{remark}

In the given definition, the horizontal boundary is situated on the south side. In subsequent proofs, we will also employ a rotated version with the boundary on the north side. The last-passage value and free energy for this scenario will be denoted as $G^{\textup{N}, \rho }_{(0,0), (n,n)}$ or $\log Z^{ \textup{N}, \rho}_{(0,0), (n,n)}$, respectively. The superscript ``N" signifies north, as the horizontal boundary aligns with the line $y=n$.
\end{remark}

We record the expectations of the last-passage value and the free energy for the increment-stationary models below. For the CGM, 
\begin{equation}\label{stat_expect}
\mathbb{E}^\rho\Big[G^\rho_{{(0,0)}, (m,n)}\Big] = \frac{m}{\rho} + \frac{n}{1-\rho},
\end{equation}
and for reference, see (3.24) in the arXiv version of \cite{CGMlecture}.
For the inverse-gamma polymer, let $\Psi_0$ denote the digamma function,
\begin{equation}\label{expect_stat}
\mathbb{E}\Big[\log Z^\rho_{(0,0), (m,n)} \Big] = - m\Phi_0(\rho) -n \Phi_0(1-\rho),
\end{equation}
for reference, see (2.5) in the arXiv version of \cite{poly2}.

The subsequent variance bounds have been presented as Theorem 5.1 and Lemma 5.7 in the arXiv version of \cite{CGMlecture} for the CGM, and as Theorem 2.1 and Lemma 4.1 in the arXiv version of \cite{poly2} for the inverse-gamma polymer. We include them here for reference.

\begin{proposition}[{\cite{poly2, CGMlecture}}]\label{stat_var}
Let $0< \epsilon_0 < 1/2$. Then there exists a constant $C_1$ depending only on $\epsilon_0$ such that for $0< \rho - 1/2 \leq \epsilon_0$ and each $|\mathbf x- (n,n)|_1 \leq 1$
\begin{align*}
\Var\Big[G^\rho_{(0,0), \mathbf x}\Big] &\leq \Var\Big[G^{1/2}_{(0,0), \mathbf x}\Big] + C_1(\rho - 1/2)n\\
\Var\Big[\log Z^\rho_{(0,0), \mathbf x}\Big] &\leq \Var\Big[\log Z^{1/2}_{(0,0), \mathbf x}\Big] + C_1(\rho - 1/2)n.
\end{align*}
In addition, there exists an absolute constant $C_2$ such that 
\begin{align*}
\Var\Big[G^{1/2}_{(0,0), \mathbf x}\Big] &\leq C_2n^{2/3}\\
\Var\Big[\log Z^{1/2}_{(0,0), \mathbf x}\Big] &\leq C_2n^{2/3}
\end{align*}
\end{proposition}

In the increment-stationary model, it is observed that the geodesic or sampled polymer path tends to favor remaining on the boundary. However, for each $\rho \in (0,1)$, there exists a unique direction where the attraction effect is evenly balanced on the left and right sides of the origin. This is known as the \textit{characteristic direction}. In the increment-stationary LPP model, it is defined as 
 \begin{equation}\label{char_dir}
 \boldsymbol\xi[\rho] = (\rho^2, (1-\rho)^2).
 \end{equation}
For the increment-stationary polymer, let 
$\Psi_1$ be the trigamma functions and the characteristic direction is given by 
\begin{equation}\label{char_dir1}
\boldsymbol\xi[{{\rho}}] =  ({\Psi_1(1-{\rho})},{\Psi_1({{\rho}})}). 
\end{equation}
Note above, we use the same notation $\boldsymbol{\xi}[\rho]$, as it will be clear based on the model being discussed.

\begin{figure}[t]
\begin{center}

\tikzset{every picture/.style={line width=0.75pt}} 

\begin{tikzpicture}[x=0.75pt,y=0.75pt,yscale=-1,xscale=1]

\draw [color={rgb, 255:red, 155; green, 155; blue, 155 }  ,draw opacity=1 ][fill={rgb, 255:red, 155; green, 155; blue, 155 }  ,fill opacity=1 ]   (19.3,210.03) -- (241.7,210.03) ;
\draw  [fill={rgb, 255:red, 0; green, 0; blue, 0 }  ,fill opacity=1 ] (98.1,210.08) .. controls (98.1,209.01) and (98.97,208.13) .. (100.05,208.13) .. controls (101.13,208.13) and (102,209.01) .. (102,210.08) .. controls (102,211.16) and (101.13,212.03) .. (100.05,212.03) .. controls (98.97,212.03) and (98.1,211.16) .. (98.1,210.08) -- cycle ;
\draw  [fill={rgb, 255:red, 0; green, 0; blue, 0 }  ,fill opacity=1 ] (198.5,110.08) .. controls (198.5,109.01) and (199.37,108.13) .. (200.45,108.13) .. controls (201.53,108.13) and (202.4,109.01) .. (202.4,110.08) .. controls (202.4,111.16) and (201.53,112.03) .. (200.45,112.03) .. controls (199.37,112.03) and (198.5,111.16) .. (198.5,110.08) -- cycle ;
\draw [color={rgb, 255:red, 128; green, 128; blue, 128 }  ,draw opacity=1 ] [dash pattern={on 0.84pt off 2.51pt}]  (200.45,110.08) -- (9.82,239.38) ;
\draw [shift={(8.17,240.5)}, rotate = 325.85] [color={rgb, 255:red, 128; green, 128; blue, 128 }  ,draw opacity=1 ][line width=0.75]    (10.93,-3.29) .. controls (6.95,-1.4) and (3.31,-0.3) .. (0,0) .. controls (3.31,0.3) and (6.95,1.4) .. (10.93,3.29)   ;
\draw [color={rgb, 255:red, 155; green, 155; blue, 155 }  ,draw opacity=1 ][line width=4.5]    (32.1,210.03) -- (74.1,210.03) ;
\draw [color={rgb, 255:red, 0; green, 0; blue, 0 }  ,draw opacity=1 ][line width=2.25]  [dash pattern={on 6.75pt off 4.5pt}]  (100.05,210.08) -- (44.1,210.03) ;
\draw [color={rgb, 255:red, 0; green, 0; blue, 0 }  ,draw opacity=1 ][line width=1.5]  [dash pattern={on 5.63pt off 4.5pt}]  (44.1,210.03) .. controls (46.1,155.23) and (172.1,170.43) .. (200.45,110.08) ;

\draw (86.5,215.03) node [anchor=north west][inner sep=0.75pt]    {$(0,0)$};
\draw (209.6,101.13) node [anchor=north west][inner sep=0.75pt]    {$( n,n)$};
\draw (5.2,242.73) node [anchor=north west][inner sep=0.75pt]    {$-\boldsymbol{\xi}[\rho]$-directed};

\end{tikzpicture}
\captionsetup{width=.8\linewidth}
\caption{
An illustration of Proposition \ref{exit2} and Proposition \ref{exit3}: Given the $\rho$-boundary, the geodesic/polymer between $(0,0)$ and $(n,n)$, shown as the black dotted line,  is highly likely to adhere to the boundary and exit near the intersection point of the $(-\boldsymbol\xi[\rho])$-directed ray from $(n,n)$ and the $\mathbf{e}_1$-axis.} \label{fig4}
\end{center}
\end{figure}
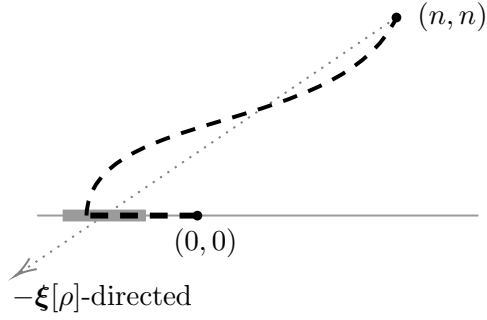

A consequence of this is that the geodesic or sampled polymer paths from $(0,0)$ to $n\boldsymbol\xi[\rho]$ spend order  $n^{2/3}$ number of steps on the boundary before taking an $\mathbf{e}_2$ step into the bulk. This is summarized in Proposition \ref{exit2} and Proposition \ref{exit3} below, and for an illustration, refer to Figure \ref{fig4}.
In the increment-stationary LPP,  the following proposition is based on \cite[Lemma 4.2]{seppcoal}.
\begin{proposition}\label{exit2}
Fix $\rho \in (0,1)$, let $i^{\rho}_*$ denote the maximizing index of $G^\rho_{(0,0), (n,n)}$, and  let $u^\rho$ denote the $\mathbf{e}_1$-coordinate of the intersection point of the $(-\boldsymbol\xi[\rho])$-directed ray from $(n,n)$ with the $\mathbf{e}_1$-axis. For each $\epsilon_0 >0$, there exist positive constants $C, n_0$ such that for $n\geq n_0$ and $t\geq 1$,
$$\mathbb{P}\Big(i^{\rho}_* \in \big({u^\rho} - \epsilon_0 tn^{2/3},{u^\rho} + \epsilon_0 tn^{2/3}\big)\Big) \geq 1-e^{-Ct^3}.$$
\end{proposition}

The version of the result for the increment stationary inverse-gamma polymer is stated below, based on \cite[Lemma 4.12]{ras-sep-she-}. 
\begin{proposition}\label{exit3}
Fix $\rho \in (0,1)$, let $u^\rho$ denote the $\mathbf{e}_1$-coordinate of the intersection point of the $(-\boldsymbol\xi[\rho])$-directed ray from $(n,n)$ with the $\mathbf{e}_1$-axis. For each $\epsilon_0 >0$, there exist positive constants $C_1, C_2 , n_0$ such that for $n\geq n_0$ and $t\geq 1$,
$$\mathbb{P}\Big(Q^\rho_{(0,0), (n,n)}\Big( \bigcup_{k\in ({u^\rho} - \epsilon_0 tn^{2/3},{u^\rho} + \epsilon_0 tn^{2/3})} \mathbb{X}_{(k,0), (n,n)}\Big) \geq 1-e^{-C_1 t^2n^{1/3}}\Big) \geq 1-e^{-C_2t^3}.$$
\end{proposition}

\subsection{Busemann functions in the CGM}\label{semiinf}

We start by defining the semi-infinite geodesics in the CGM, which is the model defined on $\mathbb{Z}^2$ without any boundary. A semi-infinite up-right path $(\mathbf 
 z_i)_{i=0}^\infty$ is a \textit{semi-infinite geodesic} if any of its finite subpaths is a geodesic, that is,  
$$\forall k<l \text{ in } \Z_{\geq0}, \text{ it holds that  }G_{\mathbf z_k, \mathbf z_l}= \sum_{i=k}^l \omega_{\mathbf z_i}.$$
For a direction $\boldsymbol\xi\in\R^2_{\ge0}\setminus\{(0,0)\}$, the semi-infinite path $(\mathbf z_i)_{i=0}^\infty$ is $\boldsymbol\xi$-\textit{directed} if $\mathbf z_i/|\mathbf z_i|_1 \rightarrow \boldsymbol \xi/|\boldsymbol\xi|_1$ as $i\rightarrow \infty$. 
In the CGM, it is natural to index 
 spatial directions $\boldsymbol\xi$  by the characteristic directions $\boldsymbol \xi[\rho]$, as defined in \eqref{char_dir}.
Following the works \cite{multicoupier, buse3, coalnew}, almost surely every semi-infinite geodesic has an asymptotic direction, and for each fixed direction $\boldsymbol\xi[\rho]$ and $\mathbf x\in\Z^2$, there is a unique $\boldsymbol\xi[\rho]$-directed semi-infinite geodesic $\bgeod{\,\rho}{\mathbf x} =  \left(\bgeod{\,\rho}{\mathbf x}_i\right)_{i=0}^\infty$ such that  $\bgeod{\,\rho}{\mathbf x}_0=\mathbf x$.

Our next proposition defines and summarizes an important property of the Busemann function in the CGM. For reference, see \cite[Theorem 3.7]{seppcoal}.
\begin{proposition} [{\cite{seppcoal}}]\label{t:buse}   Fix $\rho\in(0,1)$.  There exists a process $\{B^\rho_{\mathbf x, \mathbf y}\}_{\mathbf x, \mathbf y\in\Z^2}$, called the Busemann process, which has the following properties. 
\begin{enumerate} [{\rm(i)}] 

\item  With probability one,  for each $\mathbf x, \mathbf y \in \Z^2$, 
$$B^\rho_{\mathbf x,\mathbf y} = \lim_{n\rightarrow \infty} \bigl( G_{\mathbf x, \mathbf u_n} - G_{\mathbf y, \mathbf u_n}\bigr) 
$$
for any sequence $\mathbf u_n$ such that $\abs{\mathbf u_n}_1\to\infty$ and  $\mathbf u_n/|\mathbf u_n|_1 \rightarrow \boldsymbol\xi[\rho]/|\boldsymbol \xi[\rho]|_1$ as $n\rightarrow \infty$.

\item  Define the  dual weights by 
$$\text{$\widecheck{\omega}^\rho_{\mathbf{z}} = B^\rho_{\mathbf{z}-\mathbf{e}_1, \mathbf{z}}\wedge B^\rho_{\mathbf{z}-\mathbf{e}_2, \mathbf{z}}$ \ \ \ for $\mathbf{z} \in \Z^2$.} $$ 
 Fix a  bi-infinite nearest-neighbor  down-right path ${\boldsymbol \pi} = (\boldsymbol \pi_i)_{i\in\Z}$ on $\Z^2$. 
 This means that $\boldsymbol \pi_{i+1}- \boldsymbol \pi_i\in\{\mathbf{e}_1,  -\mathbf{e}_2\}$.  Then the random variables 
 \begin{align*}  
 &\Big\{ B^\rho_{\boldsymbol \pi_i, \boldsymbol \pi_{i+1}}: i\in\Z \Big\}, \ 
  \Big\{\omega_{\mathbf y}: 
\text{$\mathbf y\in\Z^2$ lies strictly to the left of and below ${\boldsymbol \pi}$} \Big\}, \\
&\qquad \text{and} \quad 
 \Big\{\widecheck\w^\rho_{\mathbf{z}}: 
\text{$\mathbf z\in\Z^2$ lies strictly to the right of and above ${\boldsymbol \pi}$} \Big\}
\end{align*} 
  are all mutually independent with  marginal distributions 
\be\label{buse78}  B^\rho_{\mathbf x, \mathbf x+ \mathbf e_1}\sim \Exp(\rho),
\quad  
B^\rho_{\mathbf x, \mathbf x+\mathbf e_2}\sim \Exp(1-\rho) \quad\text{and}\quad \omega_{\mathbf y}, \;\widecheck{\omega}^\rho_{\mathbf z}\sim \Exp(1).
\ee
\end{enumerate} 
\end{proposition}
To establish the connection between the Busemann function and the semi-infinite geodesics, it turns out that the unique $\boldsymbol\xi[\rho]$-directed semi-infinite geodesic from $\mathbf x$ can also be defined as below. Let $\bgeod{\rho}{\mathbf x}_0= \mathbf x$, and for $k\ge 0$, \beq\label{busegeo}  
\bgeod{\,\rho}{x}_{k+1}=\begin{cases}   \bgeod{\rho}{\mathbf x}_{k} + \mathbf e_1, &\text{if } \ B^\rho_{\bgeod{\rho}{\mathbf x}_{k},\bgeod{\rho}{\mathbf x}_{k} + \mathbf e_1} \le  B^\rho_{\bgeod{\rho}{\mathbf x}_{k},\bgeod{\rho}{\mathbf x}_{k} + \mathbf e_2}
\\[5pt]  
 \bgeod{\rho}{\mathbf  x}_{k} + \mathbf  e_2, &\text{if } \   B^\rho_{\bgeod{\rho}{\mathbf  x}_{k},\bgeod{\rho}{\mathbf  x}_{k} + \mathbf  e_2} <  B^\rho_{\bgeod{\rho}{\mathbf  x}_{k},\bgeod{\rho}{\mathbf x}_{k} + \mathbf  e_1}. 
\end{cases} 
  \eeq
In addition to these, we can also define semi-infinite paths in the southwest direction using the Busemann functions. Let $\mathbf{b}^{{\rm sw},\rho, \mathbf x}_0=\mathbf  x,$ and for $k\ge 0$
  \beq\label{bg16} \begin{aligned}  
\mathbf{b}^{{\rm sw},\rho, \mathbf x}_{k+1}&=\begin{cases}   \mathbf{b}^{{\rm sw},\rho, \mathbf x}_{k} - \mathbf e_1, &\text{if } \ B^\rho_{\mathbf{b}^{{\rm sw},\rho, \mathbf x}_{k}-\mathbf e_1,\,\mathbf{b}^{{\rm sw},\rho, \mathbf x}_{k}}\le   B^\rho_{\mathbf{b}^{{\rm sw},\rho, \mathbf x}_{k}-\mathbf e_2,\,\mathbf{b}^{{\rm sw},\rho, \mathbf x}_{k}}
\\[6pt]  
 \mathbf{b}^{{\rm sw},\rho, \mathbf x}_{k} - \mathbf e_2, &\text{if } \   B^\rho_{\mathbf{b}^{{\rm sw},\rho, \mathbf x}_{k}-\mathbf e_2,\,\mathbf{b}^{{\rm sw},\rho, \mathbf x}_{k}} <   B^\rho_{\mathbf{b}^{{\rm sw},\rho, \mathbf x}_{k}-\mathbf e_1,\,\mathbf{b}^{{\rm sw},\rho, \mathbf x}_{k}} . 
\end{cases} 
\end{aligned} \eeq
Note that $\mathbf{b}^{{\rm sw},\rho, \mathbf x}$ is, in fact, the unique semi-infinite geodesic in the $(-\boldsymbol\xi[\rho])$-direction for the dual environment ${\wc \omega^\rho_{\mathbf{z}\in \mathbb{Z}^2}}$, as shown in \cite[Theorem 5.1]{seppcoal}. We will simply refer to them as southwest semi-infinite geodesics.

Proposition \ref{inf_stat} below connects these two kinds of semi-infinite geodesics to the geodesics of two increments-stationary LPP processes that share the same boundary, which we define below.
Fix $\rho \in (0,1)$, let $Y_j = B^\rho_{(j-1,0), (j, 0)}$ be the boundary weights along the $\mathbf{e}_1$-axis. Together with the bulk weights below the $\mathbf{e}_1$-axis, $\{\omega_{\mathbf{z}}\}_{\mathbf{z}\in \mathbb{Z}\times \mathbb{Z}_{<0}}$, we may define an increment-stationary LPP with its boundary on the north side: let $\mathbf{x}\in \mathbb{Z}\times \mathbb{Z}_{<0}$ and denote its last-passage value as $G^{\textup{N},\rho}_{\mathbf{x},(0,0)}$.
In addition to this, with the dual weights above the $\mathbf{e}_1$-axis,  $\{\wc \omega^\rho_{\mathbf{z}}\}_{\mathbf{z}\in \mathbb{Z}\times \mathbb{Z}_{>0}}$, defined in Proposition \ref{t:buse} (ii), we may define another increment-stationary LPP with south boundary:  $G^\rho_{(0,0), \mathbf{y}}$ for $\mathbf{y} \in \mathbb{Z}\times \mathbb{Z}_{>0}$. 
For an illustration of the proposition, see Figure \ref{fig2}, and this result has appeared previously in  \cite[Proposition 5.2.]{seppcoal}.

\begin{proposition}\label{inf_stat}
Fix $\mathbf{x} \in  \mathbb{Z}\times \mathbb{Z}_{<0}$, the edges of the semi-infinite geodesic  $\mathbf{b}^{\rho, \mathbf{x}}$ with at least one endpoint in  $\mathbb{Z}\times \mathbb{Z}_{<0}$  are also edges of the  geodesic of  $G^{\textup{N},\rho}_{\mathbf{x}, (0,0)}$. 
Similarly, fix $\mathbf{y} \in  \mathbb{Z}\times \mathbb{Z}_{>0}$, the edges  of the southwest semi-infinite geodesic  $\mathbf{b}^{\textup{sw},\rho, \mathbf{y}}$ with at least one endpoint in  $\mathbb{Z}\times \mathbb{Z}_{>0}$  are also edges of the geodesic of  $G^{\rho}_{(0,0), \mathbf{y}}$. 
\end{proposition}

Finally, we state a disjointness result regarding the two collections of semi-infinite geodesics. Let  $\mathbf e^*=\tfrac12(\mathbf e_1+\mathbf e_2)=(\tfrac12,\tfrac12)$ denote the shift between the lattice $\Z^2$ and its dual $\Z^{2*}=\Z^2+\mathbf e^*$. 
Shift the southwest semi-infinite geodesics to the dual lattice by defining 
\[  \mathbf{b}^{*,\rho,\mathbf z}_{k} =  \mathbf{b}^{{\rm sw},\rho,\mathbf z+\mathbf e^*}_{k}-\mathbf e^* \qquad\text{  for $\mathbf z\in\Z^{2*}$ and $k\geq0$.} \]

The following disjointedness property for the primal and dual geodesics was first observed in \cite{dual}, and it played an essential role in the study of coalescence of geodesics {\cite{dual, coalnew, seppcoal}.
\begin{proposition}[{\cite[Theorem 5.1]{seppcoal}}]\label{pd_disjoint}
For any $\rho \in (0,1)$, the collections of paths 
 $ \{\bgeod{\,\rho}{z}\}_{z\in \Z^2}$ and  $ \{\mathbf{b}^{*,\rho,z}\}_{z\in \Z^{*2}}$ almost surely never cross each other, when the steps of these paths are seen as collections of unit edges.
\end{proposition}

\subsection{Busemann functions in the inverse-gamma polymer}\label{Bus+Poly}

We start by defining the semi-infinite polymer measures in the inverse-gamma polymer, which is the model defined on $\mathbb{Z}^2$ without any boundary. For a direction $\boldsymbol\xi\in\R^2_{\ge0}\setminus\{(0,0)\}$, the semi-infinite up-right path $(\mathbf z_n)_{n=0}^\infty$ is $\boldsymbol\xi$-\textit{directed} if $\mathbf z_n/|\mathbf z_n|_1 \rightarrow \boldsymbol \xi/|\boldsymbol\xi|_1$ as $n\rightarrow \infty$. 
Fix $\mathbf{v} \in \mathbb{Z}^2$, the $\boldsymbol\xi$-directed semi-infinite polymer measure is obtained as the following weak limit 
\be\label{Q41} Q_{\mathbf{v}, \mathbf{z}_n} \rightharpoonup \Pi^{\boldsymbol\xi}_\mathbf{v} \qquad \text{ as }n \to \infty,\ee
and this weak limit exists $\mathbb{P}$-almost surely in the inverse-gamma polymer  \cite[Theorem 3.8]{Jan-Ras-20-aop}. The probability measure $\Pi^{\boldsymbol\xi}_\mathbf{v}$ is the quenched path measure of a random walk (taking $\mathbf{e}_1$- and $\mathbf{e}_2$-steps) started at $\mathbf{v}$. Furthermore, the transition probabilities are given by the Busemann functions in the model, which we will define below.

Following Theorem 4.1 from \cite{Geo-etal-15}, our next proposition defines and summarizes an important property of the Busemann function in the inverse-gamma polymer. 

\begin{proposition}[{\cite[Theorem 4.1]{Geo-etal-15}}]\label{poly_buse}   Fix $\rho\in(0,1)$.  There exists a process $\{B^\rho_{\mathbf x, \mathbf y}\}_{\mathbf x, \mathbf y\in\Z^2}$, called the Busemann process, which has the following properties. 
\begin{enumerate} [{\rm(i)}] 

\item  With probability one,  for each $\mathbf x, \mathbf y \in \Z^2$, 
$$B^\rho_{\mathbf x,\mathbf y} = \lim_{n\rightarrow \infty}  \bigl(\log Z_{\mathbf x, \mathbf u_n}-\log Z_{\mathbf y, \mathbf u_n}\bigr)
$$
for any sequence $\mathbf u_n$ such that $\abs{\mathbf u_n}_1\to\infty$ and  $\mathbf u_n/|\mathbf u_n|_1 \rightarrow \boldsymbol\xi[\rho]/|\boldsymbol \xi[\rho]|_1$ as $n\rightarrow \infty$.

\item  Define the  dual weights by 
\[\wc \w^\rho_\mathbf{z}=\frac1{e^{-B^\rho_{\mathbf z-\mathbf e_1,\mathbf z}}+e^{-B^\rho_{\mathbf z-\mathbf e_2,\mathbf z}}},\quad \mathbf z\in\mathbb Z^2\,,\] 
 Fix a  bi-infinite nearest-neighbor  down-right path ${\boldsymbol \pi} = (\boldsymbol \pi_i)_{i\in\Z}$ on $\Z^2$. 
Then the random variables 
 \begin{align*}  
  &\Big\{ B^\rho_{{\boldsymbol \pi}_{i}, {\boldsymbol \pi}_{i+1}}: i\in\Z \Big\}, \ 
  \Big\{\w_{\mathbf y}: 
\text{$\mathbf y\in\Z^2$ lies strictly to the left of and below $\boldsymbol \pi$} \Big\}, \\
&\qquad \text{and} \quad 
 \Big\{\widecheck \w^\rho_{\mathbf{z}}: 
\text{$\mathbf z\in\Z^2$ lies strictly to the right of and above $\boldsymbol \pi$} \Big\}
\end{align*} 
  are all mutually independent with  marginal distributions  
\be\label{buse781}  e^{B^\rho_{\mathbf x, \mathbf x+ \mathbf e_1}}\sim \textup{Ga}^{-1}(\rho),
\quad  
e^{B^\rho_{\mathbf x, \mathbf x+\mathbf e_2}}\sim \textup{Ga}^{-1}(1-\rho) \quad\text{and}\quad \w_{\mathbf y}, \;\widecheck{\w}^\rho_{\mathbf z}\sim \textup{Ga}^{-1}(1).
\ee
\end{enumerate} 
\end{proposition}

To simplify the notation, for any $\mathbf{z}\in\mathbb Z^2$, let us define 
$$I^{\rho}_{\mathbf z}=e^{B^\rho_{\mathbf z- \mathbf e_1, \mathbf z}} \qquad J^{\rho}_{\mathbf \mathbf z}=e^{B^\rho_{\mathbf z-\mathbf e_2,\mathbf z}}.$$

Again, using the independence from the theorem above, we can define two increment-stationary models using the same horizontal boundary along $y=0$ while the environments above and below the $\mathbf{e}_1$-axis are independent:
\begin{itemize}
\item The weights 
$\{\w_{\mathbf y}\}_{\mathbf{y} \in \mathbb{Z} \times \mathbb{Z}_{<0}}$ and $\{I^{{\rho}}_{k \mathbf e_1}\}_{k\in \mathbb{Z}}$ are mutually independent, and together they define a increment-stationary polymer with north boundary on the horizontal line $y = 0$.  The partition function and quenched polymer measure will be denoted by $Z^{\textup{N},{{\rho}}}_{\bbullet, (0,0)}, Q^{{ \textup{N},{\rho}}}_{\bbullet, (0,0)}$.

\item The weights 
$\{\wc \w_{\mathbf z}\}_{\mathbf{z} \in \mathbb{Z} \times \mathbb{Z}_{>0}}$ and $\{I^{{\rho}}_{k \mathbf e_1}\}_{k\in \mathbb{Z}}$ are mutually independent, and together they define a increment-stationary polymer with south boundary on the horizontal line $y = 0$.  The partition function and quenched polymer measure will be denoted by $Z{}^{{{\rho}}}_{(0,0), \bbullet}, Q{}^{{{\rho}}}_{(0,0), \bbullet}$. 
\end{itemize}
We note that the origin $(0,0)$ can be translated to any other vertex $\mathbf v\in \mathbb{Z}^2$ in the definitions above. 

Next, we connect the Busemann functions with semi-infinite polymer measures. Recall the $\boldsymbol\xi[\rho]$-directed semi-infinite polymer measure $\Pi^{{{\rho}}}_{\mathbf v}$ defined in \eqref{Q41}, it turns out that its  (random) transition probabilities are given by  
\begin{align}
\begin{split}
\pi^{{\rho}}(\mathbf x, \mathbf x+\mathbf e_1) 
= \frac{J^{{\rho}}_{\mathbf x+\mathbf e_2}}{I^{{\rho}}_{\mathbf x+\mathbf e_1}+ J^{{\rho}}_{\mathbf x+\mathbf e_2}} \quad \text{ and } \quad 
\pi^{{\rho}}(\mathbf x, \mathbf x+\mathbf e_2) 
= \frac{I^{{\rho}}_{\mathbf x+\mathbf e_1}}{I^{{\rho}}_{\mathbf x+\mathbf e_1}+ J^{{\rho}}_{\mathbf x+\mathbf e_2}}. 
\end{split}
\label{polymerRWRE}
\end{align}
In addition to this, we could also define a $(-\boldsymbol\xi[\rho])$-directed southwest semi-infinite polymer measure starting at $\mathbf{v}$. Denoted it by $\Pi{}^{{\textup{sw}, {\rho}}}_{\mathbf v}$, and it is defined to have the following (random) transition probabilities 
\beq \pi{}^{\textup{sw}, {\rho}}(\mathbf x, \mathbf x-\mathbf e_1) = \frac{J^{{\rho}}_{\mathbf x}}{I^{{\rho}}_{\mathbf x}+ J^{{\rho}}_{\mathbf x}}\quad\text{ and }\quad \pi^{\textup{sw}, {\rho}}(\mathbf x, \mathbf x-\mathbf e_2) = \frac{I^{{\rho}}_{\mathbf x}}{I^{{\rho}}_{\mathbf x}+ J^{{\rho}}_{\mathbf x}}. \label{backpolymerRWRE}\eeq

Our next proposition relates the semi-infinite polymers to the polymer measures in the increment-stationary polymer.  
For $\mathbf u$ and $\mathbf v $ in $\mathbb Z^2$ with $\mathbf u\cdot \mathbf e_2 < \mathbf v\cdot \mathbf e_2 $. Let $\Pi^\rho_{\mathbf  u,\mathbf  v}$ be the distribution of the Markov chain that starts at $\mathbf u$,  has transition probabilities $\pi^\rho(\mathbf x,\mathbf  x+\mathbf  e_i)$ for $i\in\{1,2\}$ if $\mathbf x\cdot \mathbf e_2 < \mathbf v\cdot \mathbf e_2 $, and ends when it gets to the horizontal line $ y= \mathbf v\cdot \mathbf e_2$. Note the sampled path from $\Pi^\rho_{\mathbf u,\mathbf v}$ is just the sampled path from the semi-infinite polymer measure $\Pi^\rho_{\mathbf u}$ up until it reaches the line $ y= \mathbf v\cdot \mathbf e_2$.
Similarly, let $\wc\Pi{}^{\rho}_{\mathbf v, \mathbf u}$ be the distribution of the Markov chain that starts at $\mathbf v$,  has transition probabilities $\pi^{\textup{sw}, \rho}(\mathbf x,\mathbf x-\mathbf e_i)$ for $i\in\{1,2\}$, $\mathbf x\cdot \mathbf e_2 > \mathbf u\cdot \mathbf e_2 $, and ends when it gets to the horizontal line $ y= \mathbf u\cdot \mathbf e_2$. Again, the sampled path from $\wc\Pi{}^\rho_{\mathbf  v,\mathbf u}$ is just the sampled path from the backward semi-infinite polymer measure $\Pi{}^{\textup{sw}, \rho}_{\mathbf  v}$ up until it reaches the line $ y= \mathbf u\cdot \mathbf e_2$.

\begin{proposition}[{\cite[Proposition 5.1]{ras-sep-she-}}] \label{stat_iid}
Fix $\rho \in (0,1)$, we have $\mathbb P$-almost surely, for any $\mathbf u$ and $\mathbf v$ in $\mathbb Z^2$ with $\mathbf u\cdot \mathbf e_2 < 0 < \mathbf v\cdot \mathbf e_2 $, for any $\boldsymbol\gamma\in\mathbb{X}_{\mathbf u,\mathbf v}$ (viewed as either up-right or down-left path), 
$$\Pi^{{\rho}}_{\mathbf u, (0,0)}(\boldsymbol\gamma) =  Q^{\textup{N},{{\rho}}, }_{\mathbf u, (0,0)}(\boldsymbol\gamma)
\quad\text{and}\quad\wc\Pi{}^{{\rho}}_{\mathbf v, (0,0)}(\boldsymbol\gamma) =  Q{}^{{{\rho}}}_{(0,0), \mathbf v}(\boldsymbol\gamma).$$
\end{proposition}

In the final part, we record the disjointness property between the sampled paths of semi-infinite polymer measures \(\{\Pi^{\rho}_{\mathbf{z}} : \mathbf{z} \in \mathbb{Z}^2\}\) and southwest semi-infinite polymer measures \(\{{\Pi}{}^{\textup{sw}, \rho}_{\mathbf{z}} : \mathbf{z} \in \mathbb{Z}^2\}\), as demonstrated in Section 5.2 of \cite{ras-sep-she-}.

\begin{proposition}[{\cite[Section 5.2]{ras-sep-she-}}]\label{duality}
There exists a coupling for \(\{\Pi^{\rho}_{\mathbf{z}} : \mathbf{z} \in \mathbb{Z}^2\} \cup \{{\Pi}{}^{\textup{sw}, \rho}_{\mathbf{z}} : \mathbf{z} \in \mathbb{Z}^2\}\) on \(\mathbb{Z}^2\) such that, with \(\mathbb{P}\)-almost sure probability, the edges along any sampled northeast polymer path are disjoint from the edges of any sampled southwest polymer path when the latter is shifted by \((\mathbf{e}_1 + \mathbf{e}_2) / 2\).
\end{proposition}

\section{Proof of Theorem \ref{main1}}\label{pf_main}


\begin{figure}[t]
\begin{center}

\tikzset{every picture/.style={line width=0.75pt}} 

\begin{tikzpicture}[x=0.75pt,y=0.75pt,yscale=-1,xscale=1]

\draw  [dash pattern={on 0.84pt off 2.51pt}]  (126.89,180.01) -- (247.93,180.06) ;
\draw  [fill={rgb, 255:red, 0; green, 0; blue, 0 }  ,fill opacity=1 ] (158.58,179.91) .. controls (158.58,179.01) and (159.31,178.28) .. (160.21,178.28) .. controls (161.1,178.28) and (161.83,179.01) .. (161.83,179.91) .. controls (161.83,180.8) and (161.1,181.53) .. (160.21,181.53) .. controls (159.31,181.53) and (158.58,180.8) .. (158.58,179.91) -- cycle ;
\draw [color={rgb, 255:red, 155; green, 155; blue, 155 }  ,draw opacity=0.66 ][line width=3]    (101.15,240.45) .. controls (254.24,219.6) and (180.75,150.05) .. (220.75,120.05) ;
\draw    (101.15,240.45) .. controls (240.04,217) and (172.27,135.53) .. (212.26,92.58) ;
\draw [shift={(214.16,90.64)}, rotate = 136.06] [fill={rgb, 255:red, 0; green, 0; blue, 0 }  ][line width=0.08]  [draw opacity=0] (8.93,-4.29) -- (0,0) -- (8.93,4.29) -- cycle    ;
\draw  [fill={rgb, 255:red, 0; green, 0; blue, 0 }  ,fill opacity=1 ] (219.12,120.05) .. controls (219.12,119.15) and (219.85,118.42) .. (220.75,118.42) .. controls (221.65,118.42) and (222.37,119.15) .. (222.37,120.05) .. controls (222.37,120.95) and (221.65,121.67) .. (220.75,121.67) .. controls (219.85,121.67) and (219.12,120.95) .. (219.12,120.05) -- cycle ;
\draw  [fill={rgb, 255:red, 0; green, 0; blue, 0 }  ,fill opacity=1 ] (99.53,240.45) .. controls (99.53,239.55) and (100.25,238.82) .. (101.15,238.82) .. controls (102.05,238.82) and (102.78,239.55) .. (102.78,240.45) .. controls (102.78,241.35) and (102.05,242.07) .. (101.15,242.07) .. controls (100.25,242.07) and (99.53,241.35) .. (99.53,240.45) -- cycle ;
\draw  [fill={rgb, 255:red, 0; green, 0; blue, 0 }  ,fill opacity=1 ] (197.13,120.45) .. controls (197.13,119.55) and (197.85,118.82) .. (198.75,118.82) .. controls (199.65,118.82) and (200.38,119.55) .. (200.38,120.45) .. controls (200.38,121.35) and (199.65,122.07) .. (198.75,122.07) .. controls (197.85,122.07) and (197.13,121.35) .. (197.13,120.45) -- cycle ;
\draw  [fill={rgb, 255:red, 0; green, 0; blue, 0 }  ,fill opacity=1 ] (190.87,180.08) .. controls (190.87,179.18) and (191.6,178.45) .. (192.49,178.45) .. controls (193.39,178.45) and (194.12,179.18) .. (194.12,180.08) .. controls (194.12,180.98) and (193.39,181.7) .. (192.49,181.7) .. controls (191.6,181.7) and (190.87,180.98) .. (190.87,180.08) -- cycle ;

\draw (81.33,246.93) node [anchor=north west][inner sep=0.75pt]    {$(0,0)$};
\draw (226,111.13) node [anchor=north west][inner sep=0.75pt]    {$(n,n)$};
\draw (172.2,165.88) node [anchor=north west][inner sep=0.75pt]    {$\mathbf{w}_r$};
\draw (134.4,182.53) node [anchor=north west][inner sep=0.75pt]    {$(r,r)$};
\draw (175,110.28) node [anchor=north west][inner sep=0.75pt]    {$\mathbf{w}_{n}$};

\end{tikzpicture}
\captionsetup{width=.8\linewidth}
\caption{On the event $A_{r, n}$, the geodesic between $(0,0)$ and $(n, n)$, shown in gray, must stay to the right of the semi-infinite geodesic shown in black. By path monotonicity, the geodesic will have a large transversal fluctuation at the horizontal line $y= r$. }\label{fig1}

\end{center}
\end{figure}
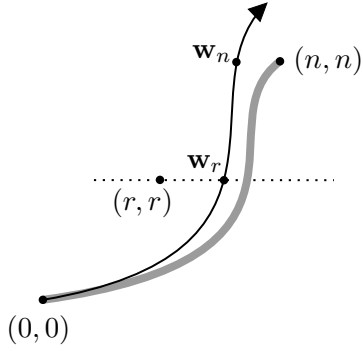
Let $\mathbf{v}_r^{\textup{min}}(G_{0,n})$ denote the lattice point with the minimum $\ell^1$-norm among the intersection points of the geodesic of $G_{(0,0),(n,n)}$ and the horizontal line $y = r$. We note that it suffices to prove the lower bound 
$$\mathbb{P}(\mathbf{v}^{\textup{min}}_r(G_{0,n}) \cdot \mathbf{e}_1-r\geq {atr^{2/3}}) \geq e^{-Ct^3}$$
for some fixed positive constant $a$, as we can then let $t'= at$ and obtain the statement in Theorem \ref{main1}. This will be done by proving a similar version of the lower bound but for the semi-infinite geodesic.

Let us look at the semi-infinite geodesic in the diagonal direction starting at ${(0,0)}$, $\mathbf{b}^{1/2, {(0,0)}} =\{\mathbf{b}^{1/2, {(0,0)}}_i\}_{i=0}^\infty$. For each integer $k$, let $\mathbf{w}_k$ denote the first point of the semi-infinite geodesic that enters the horizontal line $y=k$. For some absolute constant $a$ which we will fix later, and for any $n$ and $0\leq r\leq n/2$, define the event
$$A_{r,n} = \Big\{\mathbf{w}_r\cdot \mathbf{e}_1 -r \geq {atr^{2/3}}\Big\} \cap \Big\{\mathbf{w}_{n}\cdot \mathbf{e}_1 - n \leq -1 \Big\}.$$

As illustrated in Figure \ref{fig1}, by path monotonicity and the uniqueness of geodesics, on the event $A_{r,n}$, the finite geodesic between ${(0,0)}$ and $(n, n)$ must stay to the right of the semi-infinite geodesic $\mathbf{b}^{1/2,{(0,0)}}$. Otherwise, it must cross $\mathbf{b}^{1/2, {(0,0)}}$ before reaching $(n, n)$, and this is impossible due to the uniqueness of geodesics. 
Therefore, Proposition \ref{main_prop} below will directly imply Theorem \ref{main1}.
\begin{proposition}\label{main_prop}
there exist positive constants $C, n_0, r_0, c_0$ such that for each $n\geq n_0, r_0\leq r \leq n/2$ and  $1 \leq t \leq c_0 r^{1/3}$, 
$$\mathbb{P}(A_{r, n}) \geq e^{-Ct^3}.$$
\end{proposition}

\begin{figure}[t]
\begin{center}
\tikzset{every picture/.style={line width=0.75pt}} 

\begin{tikzpicture}[x=0.75pt,y=0.75pt,yscale=-1,xscale=1]

\draw  [dash pattern={on 0.84pt off 2.51pt}]  (126.89,180.01) -- (247.93,180.06) ;
\draw  [fill={rgb, 255:red, 0; green, 0; blue, 0 }  ,fill opacity=1 ] (158.58,179.91) .. controls (158.58,179.01) and (159.31,178.28) .. (160.21,178.28) .. controls (161.1,178.28) and (161.83,179.01) .. (161.83,179.91) .. controls (161.83,180.8) and (161.1,181.53) .. (160.21,181.53) .. controls (159.31,181.53) and (158.58,180.8) .. (158.58,179.91) -- cycle ;
\draw    (101.15,240.45) .. controls (240.04,217) and (172.27,135.53) .. (212.26,92.58) ;
\draw [shift={(214.16,90.64)}, rotate = 136.06] [fill={rgb, 255:red, 0; green, 0; blue, 0 }  ][line width=0.08]  [draw opacity=0] (8.93,-4.29) -- (0,0) -- (8.93,4.29) -- cycle    ;
\draw  [fill={rgb, 255:red, 0; green, 0; blue, 0 }  ,fill opacity=1 ] (220.13,105.22) .. controls (220.13,104.32) and (220.85,103.59) .. (221.75,103.59) .. controls (222.65,103.59) and (223.38,104.32) .. (223.38,105.22) .. controls (223.38,106.11) and (222.65,106.84) .. (221.75,106.84) .. controls (220.85,106.84) and (220.13,106.11) .. (220.13,105.22) -- cycle ;
\draw  [fill={rgb, 255:red, 0; green, 0; blue, 0 }  ,fill opacity=1 ] (99.53,240.45) .. controls (99.53,239.55) and (100.25,238.82) .. (101.15,238.82) .. controls (102.05,238.82) and (102.78,239.55) .. (102.78,240.45) .. controls (102.78,241.35) and (102.05,242.07) .. (101.15,242.07) .. controls (100.25,242.07) and (99.53,241.35) .. (99.53,240.45) -- cycle ;
\draw  [fill={rgb, 255:red, 0; green, 0; blue, 0 }  ,fill opacity=1 ] (190.87,180.08) .. controls (190.87,179.18) and (191.6,178.45) .. (192.49,178.45) .. controls (193.39,178.45) and (194.12,179.18) .. (194.12,180.08) .. controls (194.12,180.98) and (193.39,181.7) .. (192.49,181.7) .. controls (191.6,181.7) and (190.87,180.98) .. (190.87,180.08) -- cycle ;
\draw  [dash pattern={on 4.5pt off 4.5pt}]  (221.75,106.84) .. controls (199.05,177.27) and (214.71,211.03) .. (180.78,234.01) ;
\draw [shift={(179.2,235.05)}, rotate = 327.43] [color={rgb, 255:red, 0; green, 0; blue, 0 }  ][line width=0.75]    (10.93,-3.29) .. controls (6.95,-1.4) and (3.31,-0.3) .. (0,0) .. controls (3.31,0.3) and (6.95,1.4) .. (10.93,3.29)   ;
\draw  [fill={rgb, 255:red, 0; green, 0; blue, 0 }  ,fill opacity=1 ] (205,180.15) .. controls (205,179.25) and (205.73,178.52) .. (206.63,178.52) .. controls (207.52,178.52) and (208.25,179.25) .. (208.25,180.15) .. controls (208.25,181.04) and (207.52,181.77) .. (206.63,181.77) .. controls (205.73,181.77) and (205,181.04) .. (205,180.15) -- cycle ;
\draw  [dash pattern={on 0.84pt off 2.51pt}]  (401.02,179.15) -- (522.06,179.19) ;
\draw  [fill={rgb, 255:red, 0; green, 0; blue, 0 }  ,fill opacity=1 ] (432.72,179.04) .. controls (432.72,178.14) and (433.44,177.42) .. (434.34,177.42) .. controls (435.24,177.42) and (435.97,178.14) .. (435.97,179.04) .. controls (435.97,179.94) and (435.24,180.67) .. (434.34,180.67) .. controls (433.44,180.67) and (432.72,179.94) .. (432.72,179.04) -- cycle ;
\draw [line width=0.75]    (375.28,239.58) .. controls (456.13,235.47) and (454.8,209.47) .. (466.63,180.84) ;
\draw  [fill={rgb, 255:red, 0; green, 0; blue, 0 }  ,fill opacity=1 ] (494.26,104.35) .. controls (494.26,103.45) and (494.99,102.73) .. (495.88,102.73) .. controls (496.78,102.73) and (497.51,103.45) .. (497.51,104.35) .. controls (497.51,105.25) and (496.78,105.98) .. (495.88,105.98) .. controls (494.99,105.98) and (494.26,105.25) .. (494.26,104.35) -- cycle ;
\draw  [fill={rgb, 255:red, 0; green, 0; blue, 0 }  ,fill opacity=1 ] (373.66,239.58) .. controls (373.66,238.69) and (374.39,237.96) .. (375.28,237.96) .. controls (376.18,237.96) and (376.91,238.69) .. (376.91,239.58) .. controls (376.91,240.48) and (376.18,241.21) .. (375.28,241.21) .. controls (374.39,241.21) and (373.66,240.48) .. (373.66,239.58) -- cycle ;
\draw  [fill={rgb, 255:red, 0; green, 0; blue, 0 }  ,fill opacity=1 ] (465,179.21) .. controls (465,178.31) and (465.73,177.59) .. (466.63,177.59) .. controls (467.52,177.59) and (468.25,178.31) .. (468.25,179.21) .. controls (468.25,180.11) and (467.52,180.84) .. (466.63,180.84) .. controls (465.73,180.84) and (465,180.11) .. (465,179.21) -- cycle ;
\draw [color={rgb, 255:red, 155; green, 155; blue, 155 }  ,draw opacity=1 ][line width=2.25]  [dash pattern={on 6.75pt off 4.5pt}]  (495.88,104.35) .. controls (488.8,132.8) and (486.13,148.13) .. (480.76,177.65) ;
\draw  [fill={rgb, 255:red, 0; green, 0; blue, 0 }  ,fill opacity=1 ] (479.13,179.28) .. controls (479.13,178.38) and (479.86,177.65) .. (480.76,177.65) .. controls (481.66,177.65) and (482.38,178.38) .. (482.38,179.28) .. controls (482.38,180.18) and (481.66,180.9) .. (480.76,180.9) .. controls (479.86,180.9) and (479.13,180.18) .. (479.13,179.28) -- cycle ;
\draw    (434.34,179.04) -- (466.63,179.21) ;
\draw [color={rgb, 255:red, 155; green, 155; blue, 155 }  ,draw opacity=0.55 ][line width=3]  [dash pattern={on 7.88pt off 4.5pt}]  (434.34,179.04) -- (479.13,179.28) ;

\draw (81.33,246.93) node [anchor=north west][inner sep=0.75pt]    {$(0,0)$};
\draw (228.42,86.62) node [anchor=north west][inner sep=0.75pt]    {$( n,n+1)$};
\draw (173.2,163.88) node [anchor=north west][inner sep=0.75pt]    {$\mathbf{w}_r$};
\draw (134.4,181.53) node [anchor=north west][inner sep=0.75pt]    {$( r,r)$};
\draw (209.33,179.9991) node [anchor=north west][inner sep=0.75pt]    {$\mathbf{w}^{\textup{sw}}_r$};
\draw (355.47,246.07) node [anchor=north west][inner sep=0.75pt]    {$( 0,0)$};
\draw (501.55,85.75) node [anchor=north west][inner sep=0.75pt]    {$( n,n+1)$};
\draw (403.02,182.55) node [anchor=north west][inner sep=0.75pt]    {$( r,r)$};

\draw (453.02,211.55) node [anchor=north west][inner sep=0.75pt]    {$G^{\textup{N}, 1/2}_{0,r}$};

\draw (456.02,116.55) node [anchor=north west][inner sep=0.75pt]    {$G^{1/2}_{r,n}$};

\end{tikzpicture}
\captionsetup{width=.8\linewidth}
\caption{\textit{Left:} Using the disjointedness property described in Proposition \ref{pd_disjoint}, the forward geodesic must be on the left of the southwest geodesic (shown in the dotted curve), thus it must stay to the left of $(n,n)$.  \textit{Right:} An illustration of applying Proposition \ref{inf_stat} to the two semi-infinite geodesics on the right. We would then obtain two geodesics from two increment-stationary LPPs with their boundary along $y=r$.} \label{fig2}

\end{center}
\end{figure}
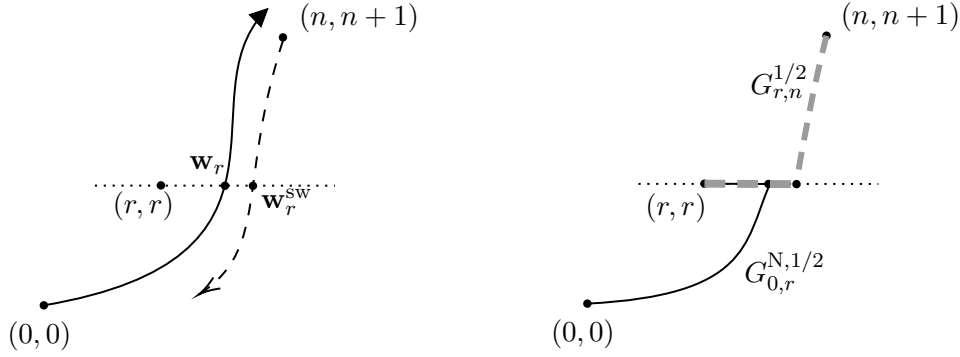

For the remainder of the section, we prove Proposition \ref{main_prop}.
To do this we will utilize the southwest (or dual) semi-infinite geodesic, which is introduced in Section \ref{semiinf}. Let $\mathbf{b}^{\textup{sw}, 1/2, (n, n+1)}$ be the southwest semi-infinite geodesic starting at $(n, n+1)$, and let $\mathbf{w}^{\textup{sw}}_r$ denote the first point where the semi-infinite geodesic $\mathbf{b}^{\textup{sw}, 1/2, (n, n+1)}$  hit the line $y=r$. Define the event 
$$B_{r,n} = \Big\{\mathbf{w}_r\cdot \mathbf{e}_1 -r \geq {atr^{2/3}}\Big\} \cap \Big\{\mathbf{w}^{\textup{sw}}_{r}\cdot \mathbf{e}_1  \geq \mathbf{w}_r\cdot \mathbf{e}_1\Big\}.$$
By Proposition \ref{pd_disjoint}, the primal geodesic $\mathbf{b}^{1/2, {(0,0)}}$ and the dual geodesic $\mathbf{b}^{\textup{sw}, 1/2, (n, n+1)} - (\tfrac{1}{2},\tfrac{1}{2})$ are (edge) disjoint. Together with path monotonicity, illustrated on the left of Figure \ref{fig2}, this implies that  
$$B_{r,n} \subset A_{r, n}.$$
Thus, it suffices to prove the lower bound $\mathbb{P}(B_{r,n}) \geq e^{-Ct^3}$.

Now, utilizing Proposition \ref{inf_stat}, we can reduce this to an estimate of two exit times for two increment-stationary LPP using the same Busemann boundary on the horizontal line $y = r$. They are denoted as $G^{\textup{N}, 1/2}_{(0,0), (r,r)} = G^{\textup{N}, 1/2}_{0,r}$ and $G^{1/2}_{(r,r), (n, n+1)} = G^{1/2}_{r, n}$ and illustrated on the right of Figure \ref{fig2}. Let us use $\tau_{0,r}$ and $\tau_{r,n}$ to denote the $\mathbf{e}_1$-coordinates of the points where the two geodesics of $ G^{\textup{N}, 1/2}_{0,r}$ and $G^{1/2}_{r, n}$ enter into the boundary on the line $y=r$. 
Fix positive constants $a_1 < a_2 < b_1$, and let us define
\begin{equation}\label{define_D}
D = D_1 \cap D_2 = \Big\{\tau_{0,r} - r \in \big[\floor{a_1 tr^{2/3}}, \floor{a_2 tr^{2/3}}\big] \Big\} \cap \Big\{\tau_{r,n} - r \in \big[\floor{b_1 t(n-r)^{2/3}}, \infty\big)\Big\}.
\end{equation}
By Proposition \ref{inf_stat}, we have $D \subset B_{r,n}$. Hence for the remainder of the section, we will prove the following proposition. 
\begin{proposition}\label{proveD}
there exist positive constants $C, n_0, r_0, c_0$ such that for each $n\geq n_0, r_0\leq r \leq n/2$ and  $1 \leq t \leq c_0 r^{1/3}$, 
$$\mathbb{P}(D) \geq e^{-Ct^3}.$$
\end{proposition}
The proof of this proposition utilizes a refined version of the Radon-Nikodym derivative trick which originally appeared in \cite{Bal-Sep-10}. This argument since then has been further developed in \cite{opt_exit, seppcoal} for the exponential LPP and \cite{ras-sep-she-} for the inverse-gamma polymer for obtaining various probability lower bounds, such as coalescence, transversal fluctuation of semi-infinite geodesics and large exit times in the increment-stationary models. 
\begin{proof}
To start, let us fix a small positive constant $\delta_0$, and its value may be lowered in the proof. We will also define 
$$a_1 = \delta_0,\quad a_2 = 100\delta_0, \quad b_1 = \sqrt{\delta_0}, \quad b_2 = 100\sqrt{\delta_0}, \quad q_1 = \delta_0, \quad q_2 = \sqrt{\delta_0},$$
these parameters will acquire their meaning in Section \ref{fix_para}. 

Let us index the weights on the boundary $y=r$ by $I_j = B^{1/2}_{(r+ j-1, r), (r+j, r)}$ for $j \in \mathbb{Z}$. Next, we will add a new set of weights obtained by lowering and raising the weights $I_j$ in two disjoint regions of the line $y=r$. Let 
\begin{equation}\label{def_para}
\lambda = \tfrac{1}{2}+q_1 tr^{-1/3} \quad  and \quad \eta = \tfrac{1}{2}-q_2 t(n-r)^{-1/3},
\end{equation}
and we may always assume $\lambda, \eta \in [\tfrac{1}{3}, \tfrac{2}{3}]$ by lowering the constant $c_0$ in Proposition \ref{proveD} according to our choices for $q_1$ and $q_2$. Define a new set of weights
\begin{alignat}{2}
\wt I_j & = \tfrac{1/2}{\lambda}I_j \sim \text{Exp}(\lambda)\qquad &&\text{ for } j \in \big[r+ \floor{a_1 tr^{2/3}}+1, r+ \floor{a_2 tr^{2/3}}\big]\nonumber\\
\wt I_j & = \tfrac{1/2}{\eta} I_j  \sim \text{Exp}(\eta)&&\text{ for } j \in   \big[r+ \floor{b_1 t(n-r)^{2/3}}+1, r+ \floor{b_2 t(n-r)^{2/3}}\big] \label{newI}\\
\wt I_j  & = I_j   & &\text{ for all other $j$}.\nonumber
\end{alignat}
We will use the notation $\wt{\mathbb{P}}^{\lambda, \eta}$ to denote the marginal distribution of the new environment with $\wt I_j$ together with the primal and dual weights below and above the line $y=r$ respectively. Let $\mathbb{P}^{1/2}$ be the marginal distribution of the original weights with $\Exp(1/2)$ on the boundary, and let $f =  d\wt{\mathbb{P}}^{\lambda, \eta}/d\mathbb{P}^{1/2}$ be the Radon-Nikodym derivative. Lemma \ref{radnik} records that 
$$\mathbb{E}^{1/2}[f^2] \leq e^{Ct^3}.$$
Now, to finish the proof, it suffices for us to show that 
\begin{equation}\label{D_est}
\wt{\mathbb{P}}^{\lambda, \eta}(D) \geq \tfrac{1}{2},
\end{equation}
because once we have this, by Cauchy-Schwarz inequality, 
$$\tfrac{1}{2} \leq \wt{\mathbb{P}}^{\lambda, \eta}(D) = \mathbb{E}^{1/2}[\mathbbm{1}_D f] \leq \sqrt{\mathbb{P}^{1/2}(D)} \sqrt{\mathbb{E}^{1/2}[f^2]}  \leq \sqrt{\mathbb{P}^{1/2}(D)} e^{Ct^3},$$
and this would finish the proof of the proposition. 

Recall the definition of $D_1$ and $D_2$ in \eqref{define_D}. Now to show \eqref{D_est}, we will show separately that $\wt{\mathbb{P}}^{\lambda, \eta}(D_1) \geq 1- Ct^{-3}$ and $\wt{\mathbb{P}}^{\lambda, \eta}(D_2) \geq 1- Ct^{-3}$. We will record them as the following two lemmas which we will prove in Section \ref{sec_lem}.
\begin{lemma}\label{lemD_1}
There exists a constant $C$ such that  $\wt{\mathbb{P}}^{\lambda, \eta}(D_1) \geq 1- Ct^{-3}$.
\end{lemma}
\begin{lemma}\label{lemD_2}
There exists a constant $C$ such that  $\wt{\mathbb{P}}^{\lambda, \eta}(D_2) \geq 1- Ct^{-3}$.
\end{lemma}
Given these two lemmas (their proofs will appear in Section \ref{sec_lem}), \eqref{D_est} holds and we have finished the proof of this proposition.
\end{proof}

\section{Proof of Theorem 1.2}\label{pf_main2}
In the positive-temperature setting, the path monotonicity used in the CGM is no longer directly available. To obtain the path monotonicity of the sampled polymer paths, we rely on a special coupling, utilizing a certain monotonicity of the quenched polymer measures.

We will start by transferring the desired lower-bound 
$$\mathbb{P}\Big(Q_{(0,0), (n,n)}\big\{ {\boldsymbol\gamma} \in \mathbb{X}_{(0,0), (n,n)} \;: \: \mathbf{e}_1 \cdot \mathbf{v}^{\textup{min}}_r ({\boldsymbol\gamma}) - r \geq atr^{2/3} \big\} \geq 1-{e^{-C_1 t^2 n^{1/3}}}\Big) \geq e^{-C_2t^3},$$
into a statement about the increment-stationary polymer model with a horizontal boundary on the line $y = n+1$. 
Let $Z^{\textup{N}, 1/2}_{(0,0), (n, n+1)}$ and $Q^{\textup{N}, 1/2}_{(0,0), (n, n+1)}$ denote the partition function and the quenched polymer measure for an increment-stationary polymer model with i.i.d.\ $\textup{Ga}^{-1}({1/2})$ 
weighted boundary on the north side along $y= n+1$. In addition, let \( Z^{\textup{N,half}}_{(0,0), (n, n+1)} \) and \( Q^{\textup{N,half}}_{(0,0), (n, n+1)} \) denote a similar model, but with all the boundary weights set to zero to the right of the point \((n, n+1)\). Recall, we may view the boundary weights as being attached to the unit edges. Finally, we note that the original bulk polymer model between $(0,0)$ and $(n,n)$ can be viewed as a polymer model with the north boundary between $(0,0)$ and $(n, n+1)$, but setting all the boundary weights to zero (note $h_0$ is still equal to $1$ by definition). We will denote these as $ Z^{\textup{N,bulk}}_{(0,0), (n, n+1)} = Z_{(0,0), (n,n)}$ and $ Q^{\textup{N,bulk}}_{(0,0), (n, n+1)} =  Q_{(0,0), (n,n)}$.

Let $\mathfrak{A}$ be the collection of paths between $(0,0)$ and the line $y=n+1$ such that these paths cross $y = r$ through or to the right of $(r+atr^{2/3}, r)$ and hit the line $y=  n+1$  at or to the left of $(n, n+1)$. In the next part, we will show the following inequality 
\begin{equation}\label{meq}
Q_{(0,0), (n,n)}\big\{ {\boldsymbol\gamma} \in \mathbb{X}_{(0,0), (n,n)} \;: \: \mathbf{e}_1 \cdot \mathbf{v}^{\textup{min}}_r ({\boldsymbol\gamma}) - r \geq atr^{2/3} \big\}  \geq Q^{\textup{N}, 1/2}_{(0,0), (n, n+1)}(\mathfrak{A}),
\end{equation}
note once we have this, to finish the proof, it remains to show the following proposition. 
\begin{proposition}\label{to_bdry}
There exist positive constants $C_1, C_2, n_0, r_0, c_0$ such that for each $n\geq n_0, r_0\leq r \leq n/2$ and  $1 \leq t \leq c_0 r^{1/3}$, 
$$
\mathbb{P}\Big(Q^{\textup{N}, 1/2}_{(0,0), (n, n+1)}(\mathfrak{A}) \geq 1-{e^{-C_1 t^2 
n^{1/3}}}\Big) \geq e^{-C_2t^3}.
$$
\end{proposition}

Now, coming back to showing \eqref{meq}, we first note that by definition 
\begin{equation}\label{meq_1}
Q_{(0,0), (n,n)}\big\{ {\boldsymbol\gamma} \in \mathbb{X}_{(0,0), (n,n)} \;: \: \mathbf{e}_1 \cdot \mathbf{v}^{\textup{min}}_r ({\boldsymbol\gamma}) - r \geq atr^{2/3} \big\} = Q^{\textup{N,bulk}}_{(0,0), (n, n+1)}(\mathfrak{A}).
\end{equation}
We also have 
\begin{equation}\label{meq_2}
Q^{\textup{N,half}}_{(0,0), (n, n+1)}(\mathfrak{A}) \geq Q^{\textup{N}, 1/2}_{(0,0), (n, n+1)}(\mathfrak{A})
\end{equation}
because these two have the same numerator, but 
$Z^{\textup{N,half}}_{(0,0), (n, n+1)} \leq Z^{\textup{N}, 1/2}_{(0,0), (n, n+1)}$ as the later has more paths with positive partition function values.  

In view of \eqref{meq_1} and \eqref{meq_2}, to establish \eqref{meq}, it remains to show that 
\begin{equation}\label{meq_3}
Q^{\textup{N,bulk}}_{(0,0), (n, n+1)}(\mathfrak{A}) \geq Q^{\textup{N,half}}_{(0,0), (n, n+1)}(\mathfrak{A}).
\end{equation}
To do this, we establish a coupling between \(Q^{\textup{N,bulk}}_{(0,0), (n, n+1)}\) and \(Q^{\textup{N,half}}_{(0,0), (n, n+1)}\). First, we recognize that both models can actually be viewed as directed polymer models defined within the rectangle \([\![(0,0), (n, n+1)]\!]\), and this is because all their boundary weights to the right of \((n, n+1)\) are zero. 
In addition, the only difference between their environments is the boundary weights along the strip \([\![(0,n+1), (n, n+1)]\!]\): in the bulk model, all weights are zero, while in the half model, the weights are i.i.d.\ \(\textup{Ga}^{-1}(1/2)\). Using Lemma A.1 of \cite{Bus-Sep-22}, for each $\mathbf z \in [\![(0,0), (n, n+1)]\!]$ the following inequalities hold, provided all the partition functions are non-zero (which assumes there exist directed paths between the starting and end points):
\begin{equation}\label{poly_m}
\frac{Z^{\textup{N,half}}_{\mathbf{z}, (n, n+1)}}{Z^{\textup{N,half}}_{\mathbf{z}+\mathbf e_1, (n, n+1)}} \geq \frac{Z^{\textup{N,bulk}}_{\mathbf{z}, (n, n+1)}}{Z^{\textup{N,bulk}}_{\mathbf{z}+\mathbf e_1, (n, n+1)}} \qquad \text{ and }\qquad \frac{Z^{\textup{N,half}}_{\mathbf{z}, (n, n+1)}}{Z^{\textup{N,half}}_{\mathbf{z}+\mathbf e_2, (n, n+1)}} \leq \frac{Z^{\textup{N,bulk}}_{\mathbf{z}, (n, n+1)}}{Z^{\textup{N,bulk}}_{\mathbf{z}+\mathbf e_2, (n, n+1)}}.
\end{equation}
These inequalities compare the ratios of partition functions in the half and bulk models when moving from \(\mathbf{z}\) to \(\mathbf{z} + \mathbf{e}_1\) and from \(\mathbf{z}\) to \(\mathbf{z} + \mathbf{e}_2\). Similar to the semi-infinite polymers, these ratios form the transition probabilities when the sampled polymer paths are viewed as random walks between $(0,0)$ and $(n, n+1)$ in random environments (see Section A.1 in the arXiv version of \cite{Jan-Ras-20-aop}).

The inequality appearing in \eqref{poly_m} implies that the sampled path in the half model is more likely to take a \(\mathbf{e}_2\)-step and less likely to take a \(\mathbf{e}_1\)-step, compared to the bulk model. Because of this, there is a coupling of \(Q^{\textup{N,bulk}}_{(0,0), (n, n+1)}\) and \(Q^{\textup{N,half}}_{(0,0), (n, n+1)}\) such that the sampled path in the half model will always be to the left of the bulk model. Now, under this coupling, because of the path monotonicity, 
$$\Big(Q^{\textup{N,half}}_{(0,0), (n, n+1)}, Q^{\textup{N,bulk}}_{(0,0), (n, n+1)}\Big)(\mathfrak A\times \mathfrak A^c) = 0.$$
With this, we obtain the desired result for \eqref{meq_3} as follows
\begin{align*}
Q^{\textup{N,half}}_{(0,0), (n, n+1)}(\mathfrak A)
& = \Big(Q^{\textup{N,half}}_{(0,0), (n, n+1)}, Q^{\textup{N,bulk}}_{(0,0), (n, n+1)}\Big)(\mathfrak A\times \mathfrak A) \geq  Q^{\textup{N,bulk}}_{(0,0), (n, n+1)}(\mathfrak A).
\end{align*}
To summarize the arguments above, we have reduced the proof of Theorem \ref{main2} to showing Proposition \ref{to_bdry}. 

To show Proposition \ref{to_bdry}, we will need to consider southwest/dual polymers. By using the coupling from Proposition \ref{duality} between primal and dual semi-infinite polymer measures and their connection to increment-stationary polymers described in Proposition \ref{stat_iid}, we can couple the following three polymer measures:
\[ \Big(Q^{\textup{N},1/2}_{(0,0), (n,n+1)}, Q^{\textup{N},1/2}_{(0,0), (r,r)},  Q{}^{1/2}_{(r,r), (n,n+1)}\Big). \]
Here, \(Q^{\textup{N},1/2}_{(0,0), (n,n+1)}\) and \(Q{}^{\textup{N},1/2}_{(0,0), (r,r)}\) correspond to the forward semi-infinite polymer starting at \((0,0)\), while \(Q{}^{1/2}_{(r,r), (n,n+1)}\) corresponds to the southwest polymer starting at \((n, n+1)\). 

For $a_1 < a_2 < b_1$, 
let us define the following collections of paths
\begin{align*}
&\mathfrak{D}_1 = \text{ paths from $(0,0)$ to $y=r$ and first hit $[\![(r + \floor{a_1tr^{2/3}}, r), (r + \floor{a_2tr^{2/3}}, r)]\!]$ }\\
&\mathfrak{D}_2 = \text{ paths from $(n,n+1)$ and first hit $y=r$ at or to the right of $(r+\floor{ b_1t(n-r)^{2/3}},r)$.}
\end{align*}
Because of the disjointness property of the sampled paths from Proposition \ref{duality}, by path monotonicity (see the left of Figure \ref{fig2} for an illustration) $$\Big(Q^{\textup{N},1/2}_{(0,0), (n,n+1)}, Q^{\textup{N},1/2}_{(0,0), (r,r)},  Q^{1/2}_{(r,r), (n,n)}\Big)(\mathfrak{A}^c \times \mathfrak{D}_1 \times \mathfrak{D}_2) = 0.$$
Then, it holds that 
\begin{align*}  Q^{\textup{N}, 1/2}_{(0,0), (n, n+1)}(\mathfrak{A})
& \geq \Big(Q^{\textup{N},1/2}_{(0,0), (n,n+1)}, Q^{\textup{N},1/2}_{(0,0), (r,r)},  Q{}^{1/2}_{(r,r), (n,n+1)}\Big)(\mathfrak{A} \times \mathfrak{D_1} \times \mathfrak{D_2})\\
& = \Big(Q^{\textup{N},1/2}_{(0,0), (r,r)},  Q{}^{1/2}_{(r,r), (n,n+1)}\Big)(\mathfrak{D_1} \times \mathfrak{D_2})
\end{align*}
To get Proposition \ref{to_bdry}, it suffices for us to show the following proposition, which corresponds to Proposition \ref{proveD} from the CGM.
\begin{proposition}
there exist positive constants $C_1, C_2, n_0, r_0, c_0$ such that for each $n\geq n_0, r_0\leq r \leq n/2$ and  $1 \leq t \leq c_0 r^{1/3}$, 
$$\mathbb{P}\Big(\Big\{Q^{\textup{N},1/2}_{(0,0), (r,r)} (\mathfrak{D}_1) \geq 1-{e^{-C_1 t^2 
n^{1/3}}}\Big\}\bigcap \Big\{Q{}^{1/2}_{(r,r), (n,n+1)}(\mathfrak{D}_2)\geq 1-{e^{-C_1 t^2 
n^{1/3}}}\Big\} \Big) \geq e^{-C_2t^3}.$$
\end{proposition}
Just like Proposition \ref{proveD}, the proof of this proposition utilizes a Radon-Nikodym derivative calculation. A simpler version of this calculation has appeared in \cite{ras-sep-she-}.   
\begin{proof}
Let us fix $\delta_0 >0$, and recall
$$a_1 = \delta_0,\quad a_2 = 100\delta_0, \quad b_1 = \sqrt{\delta_0}, \quad b_2 = 100\sqrt{\delta_0}, \quad q_1 = \delta_0, \quad q_2 = \sqrt{\delta_0}.$$ 
To simplify the notation, let us denote the weights on the boundary $y=r$ as $I_j = e^{B^{1/2}_{(r+ j-1, r), (r+j, r)}}$ for $j \in \mathbb{Z}$. We add a new set of weights that is independent of the original environment. Recall the two parameters $\lambda$ and $\eta$ defined in \eqref{def_para}, and define a new set of weights
\begin{alignat}{2}
\wt I_j^\lambda  & \sim \text{Ga}^{-1}(\lambda)\qquad &&\text{ for } j \in \big[r+ \floor{a_1 tr^{2/3}}+1, r+ \floor{a_2 tr^{2/3}}\big]\nonumber\\
\wt I_j^\eta & \sim \text{Ga}^{-1}(\eta)&&\text{ for } j \in   \big[r+ \floor{b_1 t(n-r)^{2/3}}+1, r+ \floor{b_2 t(n-r)^{2/3}}\big] \label{newI2}\\
\wt I_j  & = I_j   & &\text{ for all other $j$}.\nonumber
\end{alignat}
Again, let $\mathbb{P}^{1/2}$ and $\wt{\mathbb{P}}^{\lambda, \eta}$ to denote the marginal distributions of the original and new environments, and let $f =  d\wt{\mathbb{P}}^{\lambda, \eta}/d\mathbb{P}^{1/2}$. Lemma \ref{radnik2} records that 
$\mathbb{E}^{1/2}[f^2] \leq e^{Ct^3}.$

Let $\wt Q{}^{\textup{N},\lambda, \eta}$ and $\wt Q{}^{\lambda, \eta}$ denote the quenched polymer measures defined in the environment $\wt{\mathbb{P}}^{\lambda, \eta}$.
As we have shown below \eqref{D_est}, utilizing a Cauchy-Schwarz inequality,  to complete the proof it suffices for us to show that 
$$
{\mathbb{P}}\Big(\Big\{\wt Q{}^{\textup{N}, \lambda, \eta}_{(0,0), (r,r)} (\mathfrak{D}_1) \geq 1-{e^{-C_1 t^2 
n^{1/3}}}\Big\}\bigcap \Big\{\wt Q^{\lambda, \eta}_{(r,r), (n,n+1)}(\mathfrak{D}_2)\geq 1-{e^{-C_1 t^2 
n^{1/3}}}\Big\} \Big)\geq \tfrac{1}{2}.
$$
To do this, we will show Lemma \ref{lemD_12} and Lemma \ref{lemD_22} below. 

\begin{lemma}\label{lemD_12}
There exist constants $C_1, C_2$ such that  $${\mathbb{P}}\Big(\wt Q{}^{\textup{N}, \lambda, \eta}_{(0,0), (r,r)} (\mathfrak{D}_1) \geq 1-{e^{-C_1 t^2 
n^{1/3}}}\Big)\geq 1-C_2t^{-3}.$$
\end{lemma}
\begin{lemma}\label{lemD_22}
There exist constants $C_1, C_2$ such that  $${\mathbb{P}}\Big(\wt Q^{\lambda, \eta}_{(r,r), (n,n+1)}(\mathfrak{D}_2)\geq 1-{e^{-C_1 t^2 
n^{1/3}}} \Big)\geq 1-C_2t^{-3}.$$
\end{lemma}
Once we prove these in the next section, this will conclude the proof of this proposition.
\end{proof}

\section{Proofs of the lemmas}\label{sec_lem}

\subsection{The constants $a_1, a_2, b_1, b_2, q_1, q_2$} \label{fix_para}

\begin{figure}[t]
\begin{center}

\tikzset{every picture/.style={line width=0.75pt}} 

\begin{tikzpicture}[x=0.75pt,y=0.75pt,yscale=-0.9,xscale=0.9]

\draw    (60.32,180.2) -- (306.4,179.88) ;
\draw  [fill={rgb, 255:red, 0; green, 0; blue, 0 }  ,fill opacity=1 ] (98.53,250.54) .. controls (98.53,249.71) and (99.2,249.04) .. (100.03,249.04) .. controls (100.86,249.04) and (101.53,249.71) .. (101.53,250.54) .. controls (101.53,251.37) and (100.86,252.04) .. (100.03,252.04) .. controls (99.2,252.04) and (98.53,251.37) .. (98.53,250.54) -- cycle ;
\draw  [fill={rgb, 255:red, 0; green, 0; blue, 0 }  ,fill opacity=1 ] (287.87,60.21) .. controls (287.87,59.38) and (288.54,58.71) .. (289.37,58.71) .. controls (290.2,58.71) and (290.87,59.38) .. (290.87,60.21) .. controls (290.87,61.04) and (290.2,61.71) .. (289.37,61.71) .. controls (288.54,61.71) and (287.87,61.04) .. (287.87,60.21) -- cycle ;
\draw  [fill={rgb, 255:red, 0; green, 0; blue, 0 }  ,fill opacity=1 ] (168.53,180.17) .. controls (168.53,179.34) and (169.2,178.67) .. (170.03,178.67) .. controls (170.86,178.67) and (171.53,179.34) .. (171.53,180.17) .. controls (171.53,181) and (170.86,181.67) .. (170.03,181.67) .. controls (169.2,181.67) and (168.53,181) .. (168.53,180.17) -- cycle ;
\draw [line width=3.75]    (189.6,179.88) -- (207.6,179.88) ;
\draw [line width=3.75]    (224.4,179.88) -- (258,179.88) ;
\draw [color={rgb, 255:red, 155; green, 155; blue, 155 }  ,draw opacity=0.77 ]   (100.03,250.54) -- (218.38,165.45) ;
\draw [shift={(220,164.28)}, rotate = 144.28] [color={rgb, 255:red, 155; green, 155; blue, 155 }  ,draw opacity=0.77 ][line width=0.75]    (10.93,-3.29) .. controls (6.95,-1.4) and (3.31,-0.3) .. (0,0) .. controls (3.31,0.3) and (6.95,1.4) .. (10.93,3.29)   ;
\draw [color={rgb, 255:red, 155; green, 155; blue, 155 }  ,draw opacity=0.77 ]   (289.37,60.21) -- (232.34,202.42) ;
\draw [shift={(231.6,204.28)}, rotate = 291.85] [color={rgb, 255:red, 155; green, 155; blue, 155 }  ,draw opacity=0.77 ][line width=0.75]    (10.93,-3.29) .. controls (6.95,-1.4) and (3.31,-0.3) .. (0,0) .. controls (3.31,0.3) and (6.95,1.4) .. (10.93,3.29)   ;

\draw (103.53,253.94) node [anchor=north west][inner sep=0.75pt]    {$( 0,0)$};
\draw (129.53,159.67) node [anchor=north west][inner sep=0.75pt]    {$( r,r)$};
\draw (245.53,40) node [anchor=north west][inner sep=0.75pt]    {$(n,n+1)$};
\draw (167.62,203.61) node [anchor=north west][inner sep=0.75pt]    {$\boldsymbol{\xi}[\lambda]$-directed};
\draw (270.42,120.41) node [anchor=north west][inner sep=0.75pt]    {$-\boldsymbol{\xi}[\eta]$-directed};

\end{tikzpicture}

\captionsetup{width=.8\linewidth}
\caption{An illustration for choosing the parameters described in Section \ref{fix_para}. The two thick segments represent the intervals $\big[r+\floor{a_1 tr^{2/3}}, r+\floor{a_2tr^{2/3}}\big]$ and $\big[r+\floor{b_1t(n-r)^{2/3}},r+ \floor{b_2t(n-r)^{2/3}}\big]$. }\label{fig3}
\end{center}
\end{figure}

Recall $q_1$ and $q_2$ appeared in the definition of the boundary weight parameters $\lambda$ and $\eta$, and $a_1, a_2, b_1, b_2$ are used to define the new environments.  We will now elaborate on these constants.

For both the CGM and the inverse-gamma polymer, the idea is the same, as illustrated in Figure \ref{fig3}. Essentially, we want to select parameters such that when the \(\boldsymbol{\xi}[\lambda]\)-directed ray from \((0,0)\) and the \((- \boldsymbol{\xi}[\eta])\)-directed ray from \((n, n+1)\) cross the horizontal line \(y=r\), their \(\mathbf{e}_1\)-coordinates are within the intervals \([r +  \lfloor (a_1+\epsilon) t r^{2/3} \rfloor, r + \lfloor (a_2-\epsilon) t r^{2/3} \rfloor]\) and \([r + \lfloor  (b_1+ \epsilon) t (n-r)^{2/3} \rfloor, r + \lfloor  (b_2-\epsilon) t (n-r)^{2/3} \rfloor]\), respectively, for some \(\epsilon > 0\).

Let us fix $\delta_0 >0$, and define 
$$a_1 = \delta_0,\quad a_2 = 100\delta_0, \quad b_1 = \sqrt{\delta_0}, \quad b_2 = 100\sqrt{\delta_0}, \quad q_1 = \delta_0, \quad q_2 = \sqrt{\delta_0}.$$ 

In the CGM, looking at the $\boldsymbol{\xi}[\lambda]$-directed ray from ${(0,0)}$, ignoring the floor function, the $\mathbf{e}_1$-coordinate of its intersection point with $y=r$ is given by 
$$ r \cdot \Big(\frac{\frac{1}{2} + q_1tr^{-1/3}}{\frac{1}{2} - q_1tr^{-1/3}}\Big)^2 = r \cdot \Big(1 + \frac{2q_1tr^{-1/3}}{\frac{1}{2} - q_1tr^{-1/3}}\Big)^2.$$
Recall we let $q_1 = \delta_0$, then the above expression is bounded between 
$$[r + 2\delta_0 tr^{2/3}, r+ 99 \delta_0 tr^{2/3}]$$ when $\delta_0$ is fixed sufficiently small. A similar calculation also holds for the \((- \boldsymbol{\xi}[\eta])\)-directed ray, as well as for the inverse-gamma polymer, for which we omit the details.

Finally, we note that in our proofs of the lemmas, we may lower the values of $\delta_0$. By our definitions, the situation depicted in Figure \ref{fig3} would still hold.

\subsection{Proof of Lemma \ref{lemD_1} and Lemma \ref{lemD_2}}\label{pf_lem}
Before the proofs, we will introduce some notation. 
Recall the two increment-stationary LPP with boundary $y=r$ are denoted by $G^{\textup{N}, \rho}_{0,r}$ and $G^{\rho}_{r, n}$. For $a<b$ and $c<d$, let $G^{\rho}_{0, r}[a, b]$ and $G^{\textup{N}, \rho}_{r,n}[c, d]$ denote the last-passage values when considering the collection of paths whose starting points' \(\mathbf{e}_1\)-coordinates are contained within the intervals \([a,b]\) and \([c,d]\) along the boundary.

Next, instead of looking at different marginal distributions, we will use a generic probability measure $\mathbb{P}$ and use different notation to record different last-passage values obtained from different weights. For example, 
$G^{\textup{N}, 1/2}_{0,r}$, $G^{\textup{N},\lambda}_{0,r}$ and $G^{ \textup{N}, \eta}_{0,r}$ denote the last-passage time for the boundary weights $I_i \sim \Exp(\tfrac{1}{2}), \frac{1/2}{\lambda} I_i \sim \Exp(\lambda)$ and $\frac{1/2}{\eta}I_i \sim\Exp(\eta)$, respectively. Let $\wt G^{ \textup{N}, \lambda, \eta}_{0,r}$ to denote the last-passage value from the boundary $\wt I$ defined in \eqref{newI}. The notation for the last-passage value from $(r,r)$ to $(n,n+1)$ will follow the same superscript convention, except the superscript``${\textup{N}}$'' will be omitted.

\begin{proof}[Proof of Lemma \ref{lemD_1}]
Recall the definition of $\tau_{0,r}$ above \eqref{define_D}. To guarantee $\tau_{0,r}$ for the modified environment is inside $\big[r+\floor{a_1tr^{2/3}}, r+ \floor{a_2tr^{2/3}}\big]$ with high probability,  it suffices for us to show the following two inequalities 
\begin{align}
&\mathbb{P} \Big(\wt G^{ \textup{N}, \lambda, \eta}_{0,r}(-\infty, r+\floor{a_1tr^{2/3}}-1] < \wt G^{ \textup{N}, \lambda, \eta}_{0,r}[ r+ \floor{a_1tr^{2/3}},  r+ \floor{a_2tr^{2/3}}]\Big) \geq 1- Ct^{-3}\label{d1est1}\\
&\mathbb{P} \Big(\wt G^{ \textup{N}, \lambda, \eta}_{0,r}[ r+\floor{a_2tr^{2/3}}+1,\infty ) < \wt G^{ \textup{N}, \lambda, \eta}_{0,r}\big[ r+\floor{a_1tr^{2/3}},  r+\floor{a_2tr^{2/3}}\big]\Big) \geq 1- Ct^{-3}. \label{d1est2}
\end{align}

We will start with showing \eqref{d1est1}. 
Because how we defined our $a_1, a_2$ and $q_1$ in Section \ref{fix_para}, by Proposition \ref{exit2}, the event
\begin{equation}\label{def_U1}
U_1 = \Big\{ G^{\textup{N},\lambda}_{0,r} = G^{\textup{N},\lambda}_{0,r}[ r+\floor{a_1tr^{2/3}},  r+ \floor{a_2tr^{2/3}}]\Big\}
\end{equation}
occurs with $1-e^{-Ct^3}$ probability. Then, the following equality holds on the event $U_1$,
$$\wt G^{ \textup{N}, \lambda, \eta}_{0,r}[ r+ \floor{a_1tr^{2/3}},  r+ \floor{a_2tr^{2/3}}] = G^{\textup{N},\lambda}_{0,r} - \sum_{i=1}^{\floor{a_1tr^{2/3}}} ( 
   I_i - \tfrac{1/2}{\lambda} I_i ).$$
On the other hand, $\wt G^{ \textup{N}, \lambda, \eta}_{0,r}(-\infty,  r+ \floor{a_1tr^{2/3}}-1]  =  G^{\textup{N}, 1/2}_{0,r}(-\infty,  r+ \floor{a_1tr^{2/3}}-1] \leq G^{\textup{N}, 1/2}_{0,r}.$
With these, it holds that 
$$\eqref{d1est1} \geq \mathbb{P} \Big(\Big\{G^{\textup{N}, 1/2}_{0,r} < G^{\textup{N},\lambda}_{0,r} - \sum_{i=1}^{\floor{a_1tr^{2/3}}} (I_i  -  \tfrac{1/2}{\lambda} I_i)\Big\} \cap U_1 \Big).$$
Now, because of $\mathbb{P}(U_1) \geq 1-e^{-Ct^3}$, it suffices for us to show that 
\begin{equation}\label{D1_eq}
\mathbb{P} \Big(G^{\textup{N}, 1/2}_{0,r} < G^{\textup{N},\lambda}_{0,r} - \sum_{i=1}^{\floor{a_1tr^{2/3}}} (  I_i  - \tfrac{1/2}{\lambda} I_i) \Big) \geq 1- Ct^{-3}.
\end{equation}
A similar estimate appears in \cite{seppcoal}, which we will present here for completeness. First, we will look at the expectations of the terms inside \eqref{D1_eq} and  argue that there exist constants $c_1, c_2$ such that 
\begin{align}
\mathbb{E}[G^{\textup{N},\lambda}_{0,r}] - \mathbb{E}[G^{\textup{N}, 1/2}_{0,r}] \geq c_1 q_1 t^2 r^{1/3}\label{lowbd}\\
\mathbb{E}\bigg[\sum_{i=1}^{\floor{a_1tr^{2/3}}-1} (   1- \tfrac{1/2}{\lambda}) I_i\bigg] \leq c_2 q_1 a_1 t^2 r^{1/3}.\nonumber
\end{align}
The first inequality above follows directly from the expectation formula \eqref{stat_expect}, and this exact inequality appears as (5.53) in the arXiv version of \cite{CGMlecture}. The second inequality is a direct expectation computation of a sum of i.i.d\ exponential random variables, which we omit the calculation details.

Recall the definition $a_1 = q_1 = \delta_0$, fix $\delta_0$ sufficiently small such that 
$$c_2 a_1 q_1 = c_2 \delta_0^2 \leq  c_1\delta_0/10,$$ then the inequality in \eqref{D1_eq} holds for their expected values. Next, to give the probability lower bound, we show all the random variables in \eqref{D1_eq} are concentrated around their expectations on the scale $t^2r^{1/3}$. Let us define the events
\begin{align*}
U_2 &= \Big\{\Big|G^{\textup{N},\lambda}_{0,r} - \mathbb{E}[G^{\textup{N},\lambda}_{0,r}]\Big|  
 \leq \tfrac{c_1 \delta_0}{10} t^2 r^{1/3}\Big\}\\
U_3 &= \Big\{\Big|G^{\textup{N}, 1/2}_{0,r} - \mathbb{E}[G^{\textup{N}, 1/2}_{0,r}]\Big|  
 \leq \tfrac{c_1 \delta_0}{10} t^2 r^{1/3}\Big\}\\
U_4 &= \bigg\{\sum_{i=1}^{\floor{a_1tr^{2/3}}} ( \tfrac{1/2}{\lambda} -1) I_i - \mathbb{E}\bigg[\sum_{i=1}^{\floor{a_1tr^{2/3}}} ( 1- \tfrac{1/2}{\lambda} ) I_i\bigg] \leq \tfrac{c_1\delta_0}{10} t^2 r^{1/3} \bigg\}
\end{align*}
By Markov inequality,  $\mathbb{P}(U_2^c)$ and $\mathbb{P}(U_3^c)$ can be upper bounded in terms of the  variances
\begin{align*}
\mathbb{P}(U_2^c) \leq C\frac{\Var[G^{\textup{N},\lambda}_{0,r}]}{t^4r^{2/3}} \qquad 
\mathbb{P}(U_3^c) \leq C \frac{\Var[G^{\textup{N}, 1/2}_{0,r}]}{t^4r^{2/3}},
\end{align*}
and by Proposition \ref{stat_var} both variances are bounded by $Ctr^{2/3}$. This implies that  $\mathbb{P}(U_2)$ and  $\mathbb{P}(U_3)$ are lower bounded by $1-Ct^{-3}$. Lastly, $\mathbb{P}(U_4^c)$ can be upper bounded by $e^{-Ct}$ by standard concentration inequality for sub-exponential random variables, for reference see Section A.1 of \cite{diff_timecorr}. 
Now, because the event $U_2\cap U_3 \cap U_4$ is contained inside the event from \eqref{D1_eq}, we have finished the proof of \eqref{d1est1}.

To see \eqref{d1est2}, it suffices to show that the event's complement is contained inside $U_1^c$, whose probability is upper bounded by $e^{-Ct^3}$ by Proposition \ref{exit2}. To see this, on the complementing event of  $\eqref{d1est2}$
\begin{equation}\label{d1est2_c}
\Big\{\wt G^{ \textup{N}, \lambda, \eta}_{0,r}[ r+ \floor{a_2tr^{2/3}}+1,\infty ] \geq \wt G^{ \textup{N}, \lambda, \eta}_{0,r}\big[ r+\floor{a_1tr^{2/3}},  r+\floor{a_2tr^{2/3}}\big]\Big\},
\end{equation}let us denote the geodesic of $\wt G^{ \textup{N}, \lambda, \eta}_{0,r}[ r+\floor{a_2tr^{2/3}}+1,\infty ]$ as ${\boldsymbol\gamma}^*$. Next, we will argue that
\begin{equation}\label{gamma}
G^{\textup{N},\lambda}({\boldsymbol\gamma}^*) \geq G^{\textup{N},\lambda}_{0,r}\big[ r+\floor{a_1tr^{2/3}},  r+\floor{a_2tr^{2/3}}\big] \qquad \text{ on the event \eqref{d1est2_c}}.
\end{equation}
To see this, note changing the weights from $\Exp(1/2)$ to $\Exp(\lambda)$ for $\{I_i\}_{1\leq i\leq\floor{a_1tr^{2/3}}}$ does not affect the difference of the two passage values above, and decreasing the boundary weights $\{I_j\}_{j\geq \floor{a_2tr^{2/3}}+1}$ from $\Exp(1/2)$ and $\Exp(\eta)$ to $\Exp(\lambda)$ would only increase the passage value of $G^{\textup{N},\lambda}({\boldsymbol\gamma}^*)$. This justifies \eqref{gamma}, which shows that the event \eqref{d1est2_c} is contained inside $U_1^c$.
With this, we have shown both \eqref{d1est1}, \eqref{d1est2}, and finished the proof of our lemma. 
\end{proof}

Next, we will prove Lemma \ref{lemD_2}.

\begin{proof}[Proof of Lemma \ref{lemD_2}]
It suffices for us to show the following inequality
\begin{equation}
\mathbb{P} \Big(\wt G^{\lambda, \eta}_{r,n}(-\infty, r+\floor{b_1t(n-r)^{2/3}}-1] < \wt G^{\lambda, \eta }_{r,n}\big[r+ \floor{b_1t(n-r)^{2/3}}, r+ \floor{b_2t(n-r)^{2/3}}\big]\Big) \geq 1- Ct^{-3},\label{d2est}
\end{equation}
and the argument is similar to \eqref{d1est1}. Because how we fixed $b_1, b_2$ and $q_2$, by Proposition \ref{exit2}, the event
\begin{equation}\label{def_V1}
V_1 = \Big\{ G^{\eta}_{r,n} = G^{\eta}_{r,n}\big[r+\floor{b_1t(n-r)^{2/3}}, r+\floor{b_2t(n-r)^{2/3}}\big]\Big\}
\end{equation}
occurs with probability at least $1-e^{-Ct^3}$ . Then, the following equality holds on the event $V_1$,
\begin{align*}&\wt G^{\lambda, \eta}_{r,n}[r+ \floor{b_1t(n-r)^{2/3}}, r + \floor{b_2t(n-r)^{2/3}}] \\
&= G^{\eta}_{r,n} - \sum_{i=1}^{\floor{a_1tr^{2/3}}} 
  (\tfrac{1/2}{\eta}-1) I_i  
 - \sum_{i=\floor{a_1tr^{2/3}}+1}^{\floor{a_2tr^{2/3}}} 
    (\tfrac{1/2}{\eta} - \tfrac{1/2}{\lambda}) I_i 
     - \sum_{i=\floor{a_2tr^{2/3}}+1}^{\floor{b_1t(n-r)^{2/3}}} ( 
     \tfrac{1/2}{\eta}-1) I_i.
     \end{align*}
On the other hand, $\wt G^{\lambda, \eta}_{r,n}(-\infty, r+ \floor{b_1t(n-r)^{2/3}}-1] \leq G^{1/2}_{r,n}.$
With these and the fact that $\mathbb{P}(V_1) \geq 1-e^{-Ct^3}$, it suffices for us to show that 
\begin{align}
\mathbb{P} \bigg(G^{1/2}_{r,n} < G^{\eta}_{r,n} &- \sum_{i=1}^{\floor{a_1tr^{2/3}}} 
  (\tfrac{1/2}{\eta}-1) I_i  \label{eta_lb}\\
 & - \sum_{i=\floor{a_1tr^{2/3}}+1}^{\floor{a_2tr^{2/3}}} 
    (\tfrac{1/2}{\eta} - \tfrac{1/2}{\lambda}) I_i 
     - \sum_{i=\floor{a_2tr^{2/3}}+1}^{\floor{b_1t(n-r)^{2/3}}} ( 
     \tfrac{1/2}{\eta}-1) I_i\bigg) \geq 1- Ct^{-3}.\noindent
\end{align}
Note that the expectations of the three i.i.d\ random walk sums appearing above are upper bounded by 
\begin{align*}
\mathbb{E}\bigg[\sum_{i=1}^{\floor{a_1tr^{2/3}}} 
  (\tfrac{1/2}{\eta}-1) I_i + \sum_{i=\floor{a_1tr^{2/3}}+1}^{\floor{a_2tr^{2/3}}} 
    (\tfrac{1/2}{\eta} - \tfrac{1/2}{\lambda}) I_i +  \sum_{i=\floor{a_2tr^{2/3}}+1}^{\floor{b_1t(n-r)^{2/3}}} ( 
     \tfrac{1/2}{\eta}-1) I_i\bigg] &\leq c_3b_1q_2 t^2 n^{1/3}
\end{align*}
for some absolute constant $c_3$. Similar to \eqref{lowbd}, let $c_4$ be a constant such that $\mathbb{E}[G^{\eta}_{r,n}] - \mathbb{E}[G^{1/2}_{r, n}] \geq c_4 q_2 t^2 n^{1/3}$.
By lowering the value of $\delta_0$ further if necessary, it holds that 
$$c_3b_1q_2  = c_3 (\sqrt{\delta_0})^2\leq c_4\sqrt{\delta_0}/10.$$

Now, we have shown that the inequality in \eqref{eta_lb} holds for their expected values. Then, the concentration argument is the same as before, i.e.,  we define the events
\begin{align*}
V_2 &= \Big\{\Big|G^{1/2}_{r,n} - \mathbb{E}[G^{1/2}_{r,n}]\Big|  
 \leq \tfrac{c_4\sqrt{\delta_0}}{10} t^2 n^{1/3}\Big\}\\
V_3 &= \Big\{\Big|G^{\eta}_{r,n} - \mathbb{E}[G^{\eta}_{r,n}]\Big|  
 \leq \tfrac{c_4\sqrt{\delta_0}}{10} t^2 n^{1/3}\Big\}\\
V_4 &= \bigg\{\sum_{i=1}^{\floor{a_1tr^{2/3}}} 
  (\tfrac{1/2}{\eta}-1) I_i  - \sum_{i=\floor{a_1tr^{2/3}}+1}^{\floor{a_2tr^{2/3}}} 
    (\tfrac{1/2}{\eta} - \tfrac{1/2}{\lambda}) I_i 
     - \sum_{i=\floor{a_2tr^{2/3}}+1}^{\floor{b_1t(n-r)^{2/3}}} ( 
     \tfrac{1/2}{\eta}-1) I_i \\
     &- \mathbb{E}\bigg[\sum_{i=1}^{\floor{a_1tr^{2/3}}} 
  (\tfrac{1/2}{\eta}-1) I_i  - \sum_{i=\floor{a_1tr^{2/3}}+1}^{\floor{a_2tr^{2/3}}} 
    (\tfrac{1/2}{\eta} - \tfrac{1/2}{\lambda}) I_i 
     - \sum_{i=\floor{a_2tr^{2/3}}+1}^{\floor{b_1t(n-r)^{2/3}}} ( 
     \tfrac{1/2}{\eta}-1) I_i\bigg]
     \leq \tfrac{c_4\sqrt{\delta_0}}{10} t^2 n^{1/3} \bigg\}
\end{align*}
and their estimates are the same as those we derived for \(U_2\), \(U_3\), and \(U_4\) in the proof of Lemma \ref{lemD_1}. Thus we may omit this last part and conclude that the probability lower bound \eqref{eta_lb} holds.

\end{proof}

\subsection{Proof of Lemma \ref{lemD_12} and Lemma  \ref{lemD_22}}
We will again adapt the notation that for $a<b$ and $c<d$, $Z^{\rho}_{0, r}[a, b]$ and $Z^{\textup{N}, \rho}_{r,n}[c, d]$ denote the partition functions by only summing over paths whose $\mathbf{e}_1$-coordinates of the starting points on the boundary $y=r$ lie inside the intervals $[a,b]$ and $[c,d]$ respectively. 

In addition, let $\{I^\lambda\}_{j\in \mathbb{Z}}$ and $\{I^\eta\}_{j\in \mathbb{Z}}$ be the boundary weights along $y=r$  with $\textup{Ga}^{-1}(\lambda)$ and $\textup{Ga}^{-1}(\eta)$ distributions.  And let 
$Z^{\textup{N}, 1/2}_{0,r}$, $Z^{\textup{N},\lambda}_{0,r}$ and $Z^{ \textup{N}, \eta}_{0,r}$ denote the partition functions for the boundary weights with distributions $\text{Ga}^{-1}(1/2), \text{Ga}^{-1}(\lambda)$ and $\text{Ga}^{-1}(\eta)$, respectively.  The notation for the free energy from $(r,r)$ to $(n,n+1)$ will follow the same superscript convention, except the superscript``${\textup{N}}$'' will be omitted.

\begin{proof}[Proof of Lemma \ref{lemD_12}]
To start, it suffices for us to show the following two inequalities 
\begin{align}
&\mathbb{P} \Big(\log \wt Z^{ \textup{N}, \lambda, \eta}_{0,r}(-\infty, r+\floor{a_1tr^{2/3}}-1] \nonumber\\
& \qquad \qquad -  \log \wt Z^{ \textup{N}, \lambda, \eta}_{0,r}[ r+ \floor{a_1tr^{2/3}},  r+ \floor{a_2tr^{2/3}}] \leq -C't^2 r^{1/3} \Big) \geq 1- Ct^{-3}\label{d1est11}\\
&\mathbb{P} \Big(\log \wt Z^{ \textup{N}, \lambda, \eta}_{0,r}[ r+\floor{a_2tr^{2/3}}+1,\infty ) \nonumber\\
& \qquad \qquad - \log \wt Z^{ \textup{N}, \lambda, \eta}_{0,r}\big[ r+\floor{a_1tr^{2/3}},  r+\floor{a_2tr^{2/3}}\big] \leq -C' t^2 r^{1/3}\Big) \geq 1- Ct^{-3}. \label{d1est21}
\end{align}
Note these two inequalities are similar to \eqref{d1est1} and \eqref{d1est2} in the CGM proof, except for the additional factor \(C't^2 r^{1/3}\) inside the inequalities. This is not problematic because, even in the CGM proofs, we have shown that on the high probability event that we constructed, the difference between the two last-passage values is actually of order \(t^2 r^{1/3}\) apart. Thus, similar arguments can be applied here, and for completeness, we give a sketch below.

Starting \eqref{d1est11}, by Proposition \ref{exit3}, the event
$$
U_1 = \Big\{ \log Z^{\textup{N},\lambda}_{0,r}[ r+\floor{a_1tr^{2/3}},  r+ \floor{a_2tr^{2/3}}] - \log Z^{\textup{N},\lambda}_{0,r} \geq \log 1/2 \Big\}
$$
occurs with $1-e^{-Ct^3}$ probability. Then, the following inequality holds on the event $U_1$,
\begin{align*}&\log \wt Z^{ \textup{N}, \lambda, \eta}_{0,r}[ r+ \floor{a_1tr^{2/3}},  r+ \floor{a_2tr^{2/3}}]\\
&= \log Z^{\textup{N},\lambda}_{0,r}[ r+\floor{a_1tr^{2/3}},  r+ \floor{a_2tr^{2/3}}]- \sum_{i=1}^{\floor{a_1tr^{2/3}}} ( 
   \log I_i -  \log I^\lambda_i )\\
&\geq \log Z^{\textup{N},\lambda}_{0,r} + \log 1/2 - \sum_{i=1}^{\floor{a_1tr^{2/3}}} ( 
   \log I_i -  \log I^\lambda_i ) \qquad \text{ on the event $U_1$}.
\end{align*}
Together with the fact that 
$$\log \wt Z^{ \textup{N}, \lambda, \eta}_{0,r}(-\infty,  r+ \floor{a_1tr^{2/3}}-1]  =  \log Z^{\textup{N}, 1/2}_{0,r}(-\infty,  r+ \floor{a_1tr^{2/3}}-1] \leq \log Z^{\textup{N}, 1/2}_{0,r},$$
it holds that 
$$\eqref{d1est11} \geq \mathbb{P} \Big(\Big\{\log Z^{\textup{N}, 1/2}_{0,r} < \log Z^{\textup{N},\lambda}_{0,r} + \log 1/2 - \sum_{i=1}^{\floor{a_1tr^{2/3}}} (\log I_i  -  \log I^\lambda_i) - C't^2 r^{1/3}\Big\} \cap U_1 \Big).$$
Now, because of $\mathbb{P}(U_1) \geq 1-e^{-Ct^3}$ and absorbing the $\log 1/2$ term with $-C't^2 r^{1/3}$, it suffices for us to show that 
$$
\mathbb{P} \Big(\log Z^{\textup{N}, 1/2}_{0,r} < \log Z^{\textup{N},\lambda}_{0,r} - \sum_{i=1}^{\floor{a_1tr^{2/3}}} ( \log I_i  - \log  I^\lambda_i) - C' t^2 r^{1/3} \Big) \geq 1- Ct^{-3}.
$$
Using the explicit formula of the expectations, there exists a positive constant $c_1$ such that 
\begin{align*}
\mathbb{E}[\log Z^{\textup{N},\lambda}_{0,r}] - \mathbb{E}[\log Z^{\textup{N}, 1/2}_{0,r}] \geq c_1 \delta_0 t^2 r^{1/3}\\
\mathbb{E}\bigg[\sum_{i=1}^{\floor{a_1tr^{2/3}}} (\log I_i - \log I^\lambda_i) \bigg] \leq \frac{c_1 \delta_0 }{10} t^2 r^{1/3},\nonumber
\end{align*}
provided that $a_1 = q_1 = \delta_0$ is fixed sufficiently small. 
Letting \( C' = \frac{c_1\delta_0}{10} \), we can define \( U_2, U_3, U_4 \) similarly to the proof of Lemma \ref{lemD_1}. The same concentration inequalities can be applied, allowing us to obtain \eqref{d1est11}.

To see \eqref{d1est21}, we will upper bound the probability of the complement of the event in $\eqref{d1est21}$, i.e.
\begin{equation}\label{2poly}
\Big\{\log \wt Z^{\textup{N}, \lambda, \eta}_{0,r}[ r+ \floor{a_2tr^{2/3}}+1,\infty ] - \log \wt Z^{ \textup{N}, \lambda, \eta}_{0,r}\big[ r+\floor{a_1tr^{2/3}},  r+\floor{a_2tr^{2/3}}\big] > - C't^2 r^{1/3}\Big\}.
\end{equation}
To do this, we will argue that this event implies that 
$$\Big\{\log \wt Z^{\textup{N}, \lambda}_{0,r}[ r+ \floor{a_2tr^{2/3}}+1,\infty ] - \log \wt Z^{ \textup{N}, \lambda}_{0,r}\big[ r+\floor{a_1tr^{2/3}},  r+\floor{a_2tr^{2/3}}\big] > - C't^2 r^{1/3}\Big\},$$
whose probability is upper bounded by $e^{-Ct^3}$ by Proposition \ref{exit3}.

To see this inclusion, first, we note that because of the monotone coupling of the gamma random variables, we may assume that 
\begin{equation}\label{mono_para}
I^\lambda_j \leq I^{1/2}_j \leq I^{\eta}_j.
\end{equation}
First, changing the weights from \(\{I_i\}_{1\leq i \leq \floor{a_1 t r^{2/3}} }\) to \(\{I_i^\lambda\}_{1\leq i \leq \floor{a_1 t r^{2/3}} }\) does not affect the difference of the two free energies in \eqref{2poly}. We note that in this step, \(\log \wt Z^{\textup{N}, \lambda, \eta}_{0, r}[r + \floor{a_1 t r^{2/3}}, r + \floor{a_2 t r^{2/3}}]\) becomes \(\log \wt Z^{\textup{N}, \lambda}_{0, r}[r + \floor{a_1 t r^{2/3}}, r + \floor{a_2 t r^{2/3}}]\).
Second, changing the boundary weights \(\{I_j\}_{j \geq \floor{a_2 t r^{2/3}}+1}\) and \(\{I_j^\eta\}_{j \geq \floor{a_2 t r^{2/3}}+1}\) to \(\{I_j^\lambda\}_{j \geq \floor{a_2 t r^{2/3}}+1}\) in the definition of \(Z^{\textup{N}, \lambda, \eta}_{0, r}[r + \floor{a_2 t r^{2/3}} + 1, \infty]\) would only increase the free energy due to \eqref{mono_para}.
After both steps, \(Z^{\textup{N}, \lambda, \eta}_{0, r}[r + \floor{a_2 t r^{2/3}} + 1, \infty]\) now becomes \(Z^{\textup{N}, \lambda}_{0, r}[r + \floor{a_2 t r^{2/3}} + 1, \infty]\). 
With this, we have shown both \eqref{d1est11} and \eqref{d1est21}, thus finishing the proof of our lemma.
\end{proof}

Finally, for Lemma \ref{lemD_22}, 
it suffices for us to show
\begin{align*}
&\mathbb{P} \Big(\log \wt Z^{\lambda, \eta}_{r,n}(-\infty, r+\floor{b_1t(n-r)^{2/3}}-1] \\
& \qquad \qquad \qquad  -  \log \wt Z^{\lambda, \eta }_{r,n}\big[r+ \floor{b_1t(n-r)^{2/3}}, r+ \floor{b_2t(n-r)^{2/3}}\big] 
\leq - C' t^2 (n-r)^{1/3}\Big) \geq 1- Ct^{-3}.
\end{align*}
Since this is essentially the same as \eqref{d1est11}, and we have presented the argument in the CGM, we omit the details here.

\appendix
\section{Appendix}

We record two lemmas about the $L^2$ bounds for the Radon-Nikodym derivative when perturbing i.i.d.\ exponential and inverse-gamma random variables. 
\begin{lemma} \label{radnik} 
Let  $a_1, a_2>0$, $b\in \mathbb{R}$  and $r, n\in\Z_{>0}$ with $r \leq n$.  Assume that 
$$
|b|<\epsilon_0 r^{1/3}
$$
for some small positive $\epsilon_0$.
\begin{itemize}[noitemsep]
\item Let $Q$ be the  probability distribution on the product space $\Omega=\R^{\fl{a_1r^{2/3}} + \fl{a_2n^{2/3}}}$  under which the coordinates $X_i(\w)=\w_i$  are i.i.d.\ {\rm Exp}$(\tfrac{1}{2})$ random variables.  
\item Let $\wt Q$ denote the probability distribution such that $\omega_i \sim \Exp(\tfrac{1}{2} - br^{-1/3})$ for $i \in [1, \fl{a_1r^{2/3}}]$ and $\omega_j \sim \Exp(\tfrac{1}{2} + bn^{-1/3})$ for $j \in [\fl{a_1r^{2/3}}+1, \fl{a_1r^{2/3}} + \fl{a_2n^{2/3}} ]$, and all are independent.
\end{itemize}
Let $f$ denote the Radon-Nikodym derivative 
$  f(\w)=\tfrac{d \wt Q}{dQ}(\w). $
Then, there exists a constant $C$ such that 
\[  \mathbb E^{Q}\big[ f^2\big]   \leq e^{Cab^2}. 
\]
\end{lemma}

\begin{proof}
To simplify the notation, let $\lambda = \tfrac{1}{2} - br^{-1/3}$ and $\eta = \tfrac{1}{2} + bn^{-1/3}$. A direct calculation shows that 
\begin{align}
    E^{Q} [f^2] & = \int_\Omega \biggl(\prod_{i=1}^{\floor{a_1r^{2/3}}} \frac{\lambda e^{-\lambda \omega_i}}{\tfrac{1}{2}e^{-\omega_i/2}}\biggr)^2 \biggl(\prod_{i=1}^{\floor{a_2n^{2/3}}} \frac{\eta e^{-\eta \omega_i}}{\tfrac{1}{2}e^{-\omega_i/2}}\biggr)^2  Q(d\w) \nonumber\\
    &=\left(\frac{\lambda^2}{(1/2)^2}\int_0^\infty e^{-2(\lambda - 1/2)x} \tfrac{1}{2} e^{- x/2} dx\right)^{\floor{a_1r^{2/3}}}\left(\frac{\eta^2}{(1/2)^2}\int_0^\infty e^{-2(\eta - 1/2)x} \tfrac{1}{2} e^{- x/2} dx\right)^{\floor{a_2n^{2/3}}}\nonumber\\
    & = \left( \frac{\lambda^2}{\tfrac{1}{2} (2\lambda -\tfrac{1}{2})}\right)^{\floor{a_1r^{2/3}}}\left( \frac{\eta^2}{\tfrac{1}{2} (2\eta -\tfrac{1}{2})}\right)^{\floor{a_2n^{2/3}}}\label{prod_term}.
\end{align}
Both terms in \eqref{prod_term} above can be upper bounded by $e^{Cab^2}$, which is shown in the last math display of Lemma A.2 in \cite{seppcoal}.
\end{proof}

\begin{lemma}\label{radnik2}
Let  $a_1, a_2>0$, $b\in \mathbb{R}$  and $r, n\in\Z_{>0}$ with $r \leq n$.  Assume that 
$$
|b|<\epsilon_0 r^{1/3}
$$
for some small positive $\epsilon_0$.
\begin{itemize}[noitemsep]
\item Let $Q$ be the  probability distribution on the product space $\Omega=\R^{\fl{a_1r^{2/3}} + \fl{a_2n^{2/3}}}$  under which the coordinates $X_i(\w)=\w_i$  are i.i.d.\ $\textup{Ga}^{-1}(\tfrac{1}{2})$ random variables.  
\item Let $\wt Q$ denote the probability distribution such that $\omega_i \sim \textup{Ga}^{-1}(\tfrac{1}{2} - br^{-1/3})$ for $i \in [1, \fl{a_1r^{2/3}}]$ and $\omega_j \sim \textup{Ga}^{-1}(\tfrac{1}{2} + bn^{-1/3})$ for $j \in [\fl{a_1r^{2/3}}+1, \fl{a_1r^{2/3}} + \fl{a_2n^{2/3}} ]$, and all are independent.
\end{itemize}
Let $f$ denote the Radon-Nikodym derivative 
$  f(\w)=\tfrac{d \wt Q}{dQ}(\w). $
Then, there exists a constant $C$ such that 
\[  \mathbb E^{Q}\big[ f^2\big]   \leq e^{Cab^2}. 
\]
\end{lemma}
Similar to Lemma \ref{radnik}, this follows from Proposition A.10 of \cite{ras-sep-she-}, we omit the details here.

\bibliographystyle{amsplain}
\bibliography{time}

\end{document}